\documentclass[preprint,3p]{elsarticle}
\usepackage{stmaryrd}
\SetSymbolFont{stmry}{bold}{U}{stmry}{m}{n}
\usepackage{amsmath,bbm,mathbbol}
\usepackage{amssymb,amsfonts,amsthm}
\usepackage{tabularx,lmodern}
\usepackage{color,soul,enumitem}
\usepackage{epsfig,epstopdf,color}
\usepackage{hyperref,graphicx}
\usepackage{algorithm,algorithmic,multirow}
\usepackage{flushend}
\usepackage{verbatim}
\usepackage{caption}
\usepackage{graphicx, subfig}
\usepackage{arydshln}
\usepackage[level]{datetime}
\usepackage{float}
\usepackage{booktabs}
\usepackage{tabularx}
\usepackage{natbib}
\usepackage{relsize}
\usepackage{ulem}
\usepackage{setspace}
\usepackage{epstopdf}
\usepackage{appendix}
\usepackage{bm}

\DeclareMathOperator\supp{supp}

\newcommand{\be}{\begin{equation}}
\newcommand{\ee}{\end{equation}}
\newcommand{\ba}{\begin{array}}
\newcommand{\ea}{\end{array}}
\newcommand{\bea}{\begin{eqnarray}}
\newcommand{\eea}{\end{eqnarray}}
\newcommand{\beas}{\begin{eqnarray*}}
\newcommand{\eeas}{\end{eqnarray*}}

\newcommand{\bx}{{\textbf x}}
\newcommand{\by}{{\textbf y}}
\newcommand{\bz}{{\textbf z}}
\newcommand{\bn}{{\textbf n}}

\newcommand{\bp}{{\textbf p}}
\newcommand{\im}{\textrm i}

\newcommand{\bh}{{\textbf h}}

\newcommand{\bk}{{\textbf k}}
\newcommand{\bR}{{\textbf R}}

\renewcommand{\theequation}{\arabic{section}.\arabic{equation}} %
\newtheorem{exmp}{Example}
\newtheorem{remark}{Remark}[section]

\newtheorem{thm}{Theorem}
\newtheorem{lem}{Lemma}

\begin{document}

\title{On optimal zero-padding of kernel truncation method}

\author[tju]{Xin Liu}
\ead{liuxin\_921@tju.edu.cn}

\author[scu]{Qinglin Tang}
\ead{qinglin\_tang@scu.edu.cn}

\author[tju]{Shaobo Zhang\corref{5}}
\ead{shaobo\_zhang@tju.edu.cn}

\author[tju]{Yong Zhang}
\ead{Zhang\_Yong@tju.edu.cn}

\address[tju]{Center for Applied Mathematics,Tianjin University,Tianjin 300072, China}
\address[scu]{College of Mathematics, SiChuan University, No.24 South Section 1, Yihuan Road, ChengDu, China, 610065}
\cortext[5]{Corresponding author.}

\begin{abstract}
The kernel truncation method (KTM) is a commonly-used algorithm to compute the convolution-type nonlocal potential
$\Phi(\bx)=(U\ast \rho)(\bx), ~\bx \in {\mathbb R^d}$,
where the convolution kernel $U(\bx)$ might be singular at the origin and/or far-field and the density $\rho(\bx)$ is smooth and fast-decaying.
In KTM, in order to capture the Fourier integrand's oscillations that is brought by the kernel truncation,
one needs to carry out a zero-padding of the density, which means a larger physical computation domain and a finer mesh in the Fourier space by duality.
The empirical \textit{fourfold} zero-padding [ Vico \textit{et al} J. Comput. Phys. (2016) ] puts a heavy burden on memory requirement
especially for higher dimension problems.
In this paper, we derive the optimal zero-padding factor, that is,  $\sqrt{d}+1$, for the first time together with a rigorous proof.
The memory cost is greatly reduced to a small fraction, i.e., $(\frac{\sqrt{d}+1}{4})^d$, of what is needed in the original fourfold algorithm.
For example, in the precomputation step, a double-precision computation on a $256^3$ grid requires a minimum $3.4$ Gb memory with the optimal \textbf{threefold} zero-padding,
while the fourfold algorithm requires around $8$ Gb where the reduction factor is $\frac{37}{64}\approx \frac{3}{5}$.
Then, we present the error estimates of the potential and density in $d$ dimension.
Next, we re-investigate the optimal zero-padding factor for the anisotropic density.
Finally, extensive numerical results are provided to confirm the accuracy, efficiency,
optimal zero-padding factor for the anisotropic density, together with some applications to different types of nonlocal potential, including
the 1D/2D/3D Poisson, 2D Coulomb, quasi-2D/3D Dipole-Dipole Interaction and 3D quadrupolar potential.
\end{abstract}

\begin{keyword}
convolution-type nonlocal potential, kernel truncation method, optimal zero-padding, error estimates, anisotropic density
\end{keyword}

\maketitle

\section{Introduction}
In science and engineering community, nonlocal potentials, which are given by a convolution of a translational invariant Green's function with a
fast-decaying smooth function, are quite universe and common, e.g., the Newtonian potential in cosmology, the Poisson potential in electrostatics,
plasma physics and quantum physics, and are often used to account for long-range interaction.
In this article, we consider the following convolution-type nonlocal potential
\be\label{convolution}
\Phi(\bx)=[U \ast \rho](\bx) =\int_{\mathbb{R}^d} U(\bx-\by) \rho(\by) {\rm d} \by,\quad  \bx \in \mathbb{R}^d,
\ee
where $\ast$ is the convolution operator, $\bx = (x_1,\ldots, x_d)\in {\mathbb{R}^d}$, $ d = 1,2,3$ is the spatial dimension.
The potential can also be rewritten in Fourier integral
\be
\Phi(\bx)=\frac{1}{(2 \pi)^d} \int_{\mathbb{R}^d} \widehat{U}(\bk) \widehat{\rho}(\bk) e^{i \bk \cdot \bx} \rm{d} \bx,\quad  \bx \in \mathbb{R}^d,
\label{FourierForm}
\ee
where $\widehat{f}(\bk)= \int_{\mathbb{R}^d} f(\bx) e^{-i \bk \cdot \bx} {\rm d} \bx$ is the Fourier transform of function $f(\bx)$.
The density $\rho(\bx)$ is a fast-decaying smooth function and the convolution kernel $U(\bx)$ is usually singular at the origin and/or at the far field.
The Fourier transform of convolution kernel $\widehat U(\bk)$ is singular too, and sometimes the singularity becomes even stronger,
e.g., the 2D Poisson kernel $-\frac{1}{2\pi}\ln(|\bx|)$.

Since the density decays fast enough, it is reasonable to assume that the density is numerically compactly supported in a bounded domain
$\Omega \subset \mathbb R^d$, that is, $\supp \{\rho \} \subset \Omega$.
In computational practice, we first truncate the whole space into a rectangular domain, i.e., $\prod_{j=1}^d [-L_j,L_j]$,
and discretize it by $N_j\in 2\mathbb Z^{+}$ equally spaced points in the $j$-th direction, where the uniform mesh grid is
\begin{equation}
\mathcal T := \big\{ (x_{1}, x_{2},\ldots, x_{d}) \Big|~ x_{j} = \left(-\frac{N_j}{2},\ldots, \frac{N_j}{2}-1 \right)~\frac{2L_j}{N_j}, ~j = 1, \ldots, d.\big\}
\label{UniformMesh}
\end{equation}
The numerical problem is to compute the potential $\Phi$ on the mesh grid $\mathcal T$,
with discrete density values given on the same mesh grid, that is,
\bea
\{\rho(\bx_i)\}_{\bx_i \in \mathcal T} \longmapsto \{\Phi(\bx_i)\}_{\bx_i \in \mathcal T}.
\eea

As is well-known, the density function $\rho(\bx)$ is well approximated by Fourier spectral method with spectral accuracy, and it is implemented by discrete Fast Fourier Transform (FFT)\cite{ShenBook,NickBook}.
It is clear that direct discretization of either \eqref{convolution} or \eqref{FourierForm} shall encounter the singularity because the source and target grid coincides,
therefore, it requires special care to deal with the singularity in order to guarantee accuracy.
During the last ten years, there have been several fast spectral methods developed based on the Fourier spectral method, such as
the NonUniform Fast Fourier Transform method (NUFFT) \cite{NUFFT}, Gaussian-Summation method (GauSum)\cite{GauSum}, Kernel Truncation Method(KTM) \cite{FastConvGreengard}, and
Anisotropic Truncated Kernel Method (ATKM) \cite{atkm}, where the singularity is suitably treated with windowed function or kernel truncation technique in physical/Fourier space.
All these methods are spectrally accurate and efficient with a $\mathcal{O}(N_{\rm tot} \log(N_{\rm tot} ))$ complexity, where $N_{\rm tot} :=\!\prod_{j=1}^d\! N_j$ is total number of grid points.
Among them, KTM is the most simple one and has been widely adopted by the physics community to compute the nonlocal potentials\cite{Bogoliubov, CoulombCutoff}.

To compute the potential $\Phi(\bx)$ inside domain $\Omega$, the very basic idea of KTM
is to screen unnecessary interactions at the faraway distance, which leads to an effective truncated kernel $U_D(\bx)= U(\bx) \chi_D(\bx)$ with $\chi_D(\bx)$ being
the characteristic function of domain $D$,
and then to compute the new convolution $(U_D \ast \rho)$ since it coincides with the original potential inside domain of interest,
i.e.,
\bea
\Phi(\bx) = \left[U_D\ast \rho\right] (\bx),\quad ~\forall~ \bx \in \Omega~.\eea
In KTM, we choose to cut off interactions outside a large circular domain $D\subset \mathbb R^d$ by simply setting it to zero, i.e., $U(\bx)= 0$ for $\bx \notin D$,
then we apply trapezoidal rule to the resulted Fourier integral
\bea\label{FourKTM}
\Phi(\bx) = \frac{1}{(2 \pi)^d} \int_{\mathbb R^d} \widehat U_D(\bk) ~ \widehat \rho(\bk) ~ e^{i \bk \bx} ~{\rm d}\bk, \quad \bx \in \Omega,
\eea
where the Fourier integrand is smooth but oscillatory due to the Paley-Wiener theorem \cite{Paley-Wiener}.
To resolve the unpleasant oscillations caused by the kernel truncation, it requires a fine mesh in the Fourier space, or equivalently a larger physical domain.

The most natural way is to extend the density to a larger domain $S\Omega = \prod_{j=1}^d [-SL_j, SL_j]$ by zero padding,
that is, setting $\rho(\bx) = 0, \forall~ \bx\in S \Omega \setminus \Omega$. The constant $S > 1$ is called the zero-padding factor hereafter.
By duality, the Fourier mesh size is scaled down by a factor of $1/S$, and it helps to provide a better approximation so to capture the integrand's oscillations.
On the other hand, larger extension requires more storage and computational efforts and it inevitably results in a poorer performance.
Therefore, the optimal zero-padding factor is of significant importance to achieve better accuracy with the most economic efforts,
especially in high dimension.

The first result on the zero-padding factor was given by F. Vico et al in \cite{FastConvGreengard},
where the authors claimed a \textsl{fourfold} zero-padding is sufficient for a machine-precision computation.
With such fourfold padding, a typical double-precision computation on $256^3$ grid requires a minimum $8$ Gb just to store the padded density,
not to mention the auxiliary variables whose size is usually even larger,
therefore, potential evaluation of such size or any larger size seems to be an impossible task on personal computer.
Fortunately, using a similar periodicity argument as in ATKM \cite{atkm}, we can prove that
a \textbf{threefold}, instead of \textsl{fourfold}, zero-padding is \textbf{sufficient} to guarantee machine-precision accuracy in 3D,
which immediately helps reduce memory cost by a factor around $60\%$, thus make the laptop computation possible.

Another important situation is the anisotropic density case where $\rho$ is (numerically) compacted supported in an anisotropic domain.
The KTM applies readily with ease, but the memory requirement and computational costs both scale linearly with the anisotropy strength,
which makes it prohibitively expensive for strongly anisotropic density especially in high dimension.
In simulation, the potentials are usually called multiple times with the same numerical setups, for example,
in the computation  of ground state and dynamics of the nonlocal Schr\"{o}dinger equation\cite{pcg-dbec,GS_dynamic},
then it is worth the efforts to precompute any possible quantities once for all for sake of efficiency.

Numerically, the Fourier integral \eqref{FourKTM} is discretized by applying trapezoidal rule on a uniform mesh grid, and the resulted summation is accelerated by FFT/iFFT.
As point out in \cite{FastConvGreengard}, the trapezoidal rule boils down to a discrete convolution with grid density of length $N$,
regardless of the anisotropy strength and space dimension, and such discrete convolution can be accelerated by FFT on a doubly padded density.
The convolution tensor $T$, an accurate approximation of the interaction kernel, is determined by taking inverse discrete Fourier transform on truncated kernel's Fourier transform, and its analytical
expression is explicitly given by Eqn.\eqref{Tensor_Aniso} in Section \eqref{Sect-Tensor}.
Once the convolution tensor $T$ is available, the potential is computed with a pair of FFT/iFFT on a double-sized vector of length $2^d N_{\rm tot} $.
For isotropic density, the optimal zero-padding factor helps reduce the memory costs by a factor of $1-(\frac{S}{4})^d$,
which is quite substantial in high dimension, and alleviate the heavy burden on memory requirement and computational cost greatly.

\

The paper is organized as follows. In Section \ref{OptimalZeroPad},
firstly, we derive and prove the optimal zero-padding of the density in 1D/2D/3D.
Then, we re-investigate the anisotropic density case. Thirdly, we present the discrete convolution structure for isotropic and anisotropic cases,
together with an explicit formulation of the convolution tensor and a doubly FFT acceleration in subsection \ref{Sect-Tensor}.
Lastly, we give an rigorous error estimates for both the nonlocal potential in subsection \ref{Sec:ErrEst} and the density in \ref{AppFourier}.
Extensive numerical results are shown in Section \ref{NumericalResults} to illustrate the performance of our method in terms of accuracy and efficiency.
Finally, some conclusions are drawn in Section \ref{conclusions}.

\section{The optimal zero-padding } \label{OptimalZeroPad}
\setcounter{equation}{0}
In this section, we first focus on computation of the nonlocal potential generated by isotropic density, and discuss the anisotropic case later.
For sake of simplicity, we choose a square domain, i.e., $\bR_L = \prod_{j = 1}^d [-L,L]$, which is discretized with $N$ equally spaced grid points in each spatial direction,
and denote the mesh grid \ref{UniformMesh} as $\mathcal T_N$.
The density $\rho(\bx)$ is well approximated by the Fourier spectral method within spectral accuracy inside $\bR_L$ \cite{ShenBook,NickBook}.

The nonlocal potential \eqref{convolution} can be reformulated as follows
\bea
\Phi(\bx) \label{KTM_1}
&=&\int_{\mathbb{R}^d} U(\bx-\by) \rho(\by){\rm  d }\by
=\int_{\textbf{R}_L } U(\bx-\by) \rho(\by){ \rm d }\by, \\
&=& \int_{\bx+ \textbf{R}_{ L}} U (\by) \rho(\bx- \by){ \rm d} \by. \label{KTM_2}
\eea
As the density is compactly supported in $\bR_L$, Eqn.\eqref{KTM_2} is equivalent to the following convolution
\bea\label{KTM}
\Phi(\bx) =
\int_{\mathbf{B}_G} U (\by) \rho(\bx- \by) { \rm d } \by  , \quad \quad \bx \in \textbf{R}_L,
\eea
where ${\mathbf{B}_G}$ is a ball centered at the origin with radius $G:=\max_{\bx,\by\in \bR_L}|\bx-\by| = 2 \sqrt{d} L $ being the diameter of $\bR_L$.
The above equation holds because for any $\bx \in \bR_L, \by \in \mathbf{B}_G\backslash (\bx+ \textbf{R}_L)$, we have
\begin{equation}
   \bx-\by \notin \bR_L ~~\longrightarrow~~\rho(\bx-\by) = 0 ~~\Longrightarrow ~~ \int_{\mathbf B_G\backslash (\bx+\bR_L)} U(\by) \rho(\bx-\by)  {\rm d} \by = 0.
\end{equation}
To integrate Eqn.\eqref{KTM}, one needs to approximate the density $\rho(\bx)$ on a larger domain $\bR_{(2\sqrt{d}+1) L}$,
because for any $ \bx\in \bR_L$ and $\by \in \mathbf{B}_G$ we have
\begin{equation}
   \bx-\by \in \bR_{(2\sqrt{d}+1) L}.
\end{equation}
A natural way is to extend the density
to $\bR_{(2\sqrt{d}+1)L}$ by zero-padding and apply the Fourier spectral method then. However, this is not the most economic way.
As is known, the Fourier series can be simultaneously extended periodically to the whole space,
therefore, the zero-padded domain does not necessarily need to cover the concerned domain directly.

\

For the density, the finite Fourier series approximation on $\bR_{SL}$, implemented right after the $S$-fold zero-padding, together with its periodic extension,
guarantees a spectral approximation within domain $\bR_{(2S-1) L}$.
To make it clear, we present a graphical illustration in Figure \eqref{periodicExtension} to show that the twofold(left)
and threefold(right) zero-padding of the density (blue-solid line) and the periodic extensions (green-dashed line), where we take the length $L = 1$ for simplicity.
Let us take the 1D case as an example. From Figure \ref{periodicExtension}, it is clear that the periodic extension of a twofold zero-padding
coincides with the density within $\bR_{(2\times 2 -1) L} = \bR_{3L}$.
While, without zero-padding, periodic extension of the Fourier series approximation, that is done over the original interval $\bR_L$,
mismatches the density $\rho(\bx)$ in $[L,3L]$ with an non-negligible error.
Then we conclude that a \textbf{twofold} zero-padding suffices to guarantee a spectral approximation in 1D.
The same argument applies to other spatial direction, therefore, without difficulty, we can see that \textbf{threefold} zero-padding suffices
spectral approximation in the 2D and 3D case.

\

In fact, once the following condition \bea
2 \sqrt{d}+1\le 2S-1 \Longrightarrow  \sqrt{d}+1 \le S \eea
is satisfied, the $S$-fold zero-padding is both \textbf{sufficient} and \textbf{necessary}, therefore,
we derive the optimal zero-padding factor as
\begin{equation} \label{OptZeroPad}
  \scalebox{1.2}{\boxed{ \quad \quad  S = \sqrt{d} + 1.\quad \quad }}
\end{equation}
In practice, it is convenient to choose
\begin{equation} \label{OptZeroPad:Num}
  S = \lceil \sqrt{d} + 1\rceil =
         \left\{\begin{array}{ll}
               2, & d =1,\\
      3, & d =2, 3, \end{array}\right.
\end{equation}
where $ \lceil ~ \rceil$ is rounding up to the nearest integer.

\begin{remark}
The periodic extension observation holds in each spatial dimension, therefore, conclusions on optimal zero-padding factor holds true
for the general rectangular domains and different grid point numbers, that is, the spatial extend $[-L_j, L_j]$ and grid number $N_j$ along the $j$-th spatial direction
are not necessarily the same.
\end{remark}

\begin{remark}
In practice, for sake of a even better efficiency, we can also consider a fractional zero-padding, e.g., $S = \frac{5}{2}$, as long as $N_j S \in 2\mathbb Z^{+}$
with $j=1, 2, 3.$ 
 Take the 2D case as an example, the analytical optimal zero-padding factor $\sqrt{2}+1 \approx 2.4$ is close to $2.5$ but smaller than $\lceil\!\sqrt{2}\!+\!1\!\rceil=3$,
we may choose a $2.5$-fold zero-padding for a smaller memory cost and better efficiency.
\end{remark}

\begin{remark}
In standard KTM \cite{CoulombCutoff,FastConvGreengard}, one chooses to integrate the corresponding Fourier integral of \eqref{KTM} by applying the trapezoidal rule quadrature,
but it was not clear about how to choose the optimal mesh size so as to capture the integrand's oscillations.
For the first time, we derive the optimal zero-padding factor and point out that the typical \textbf{fourfold} zero-padding is only sufficient but not necessary.
The optimal zero-padding factor \eqref{OptZeroPad} helps reduce the memory cost by a factor $1-(\frac{S}{4})^d$, and the memory reduction is quite significant in higher space dimensions.
For example, in the precomputation step of \eqref{FourTransTrunKer}, a double-precision computation on a $256^3$ grid requires a minimum $3.4$ Gb memory with the optimal \textsl{threefold} zero-padding
while the fourfold original algorithm requires around $8$ Gb,  with the reduction factor being around $60\%$.

\end{remark}

\begin{remark}
The isotropic kernel truncation
facilitates the computation of the corresponding Fourier transform, which is exactly what the classical KTM adopted\cite{CoulombCutoff,FastConvGreengard}.
It is worthy to point out that the isotropic kernel truncation is not optimal in terms of efficiency, simply because the geometry of the computation domain is not taken into account.
Anisotropic kernel extension to a larger rectangular domain produces better efficiency, and we refer the readers to \cite{atkm} for more details.
\end{remark}

\begin{figure}
\centering
\includegraphics[scale=0.55]{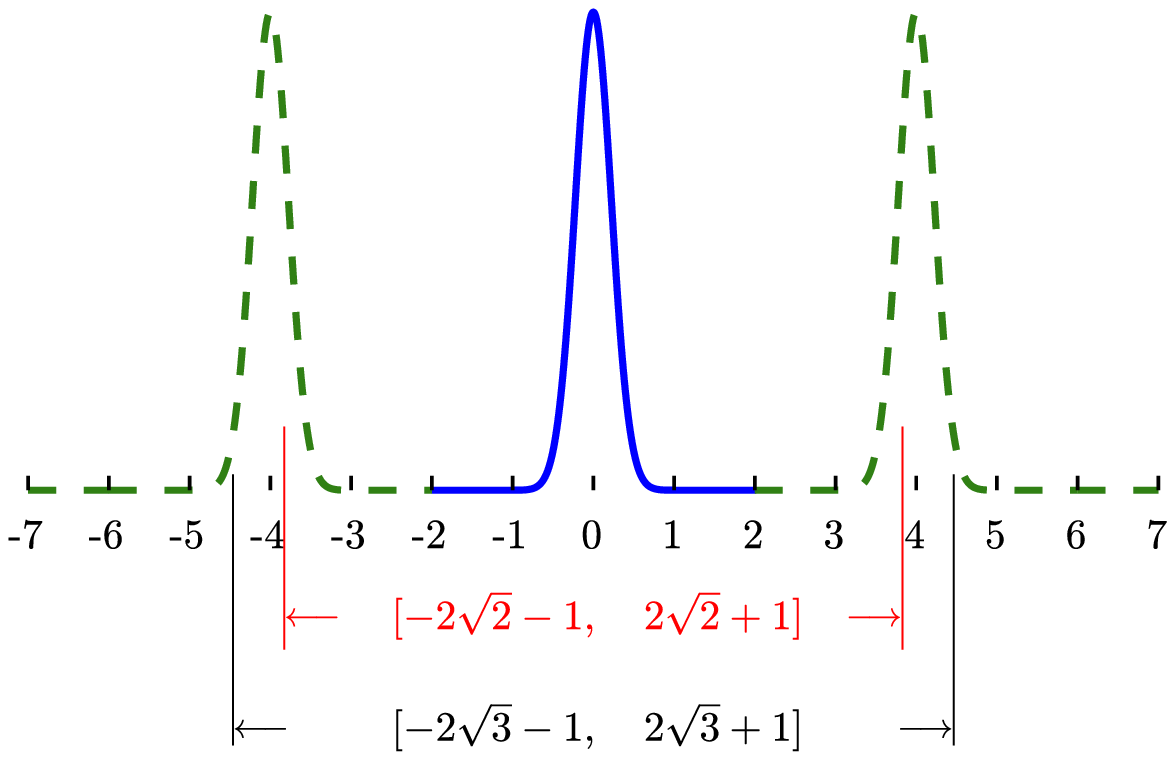}
\includegraphics[scale=0.55]{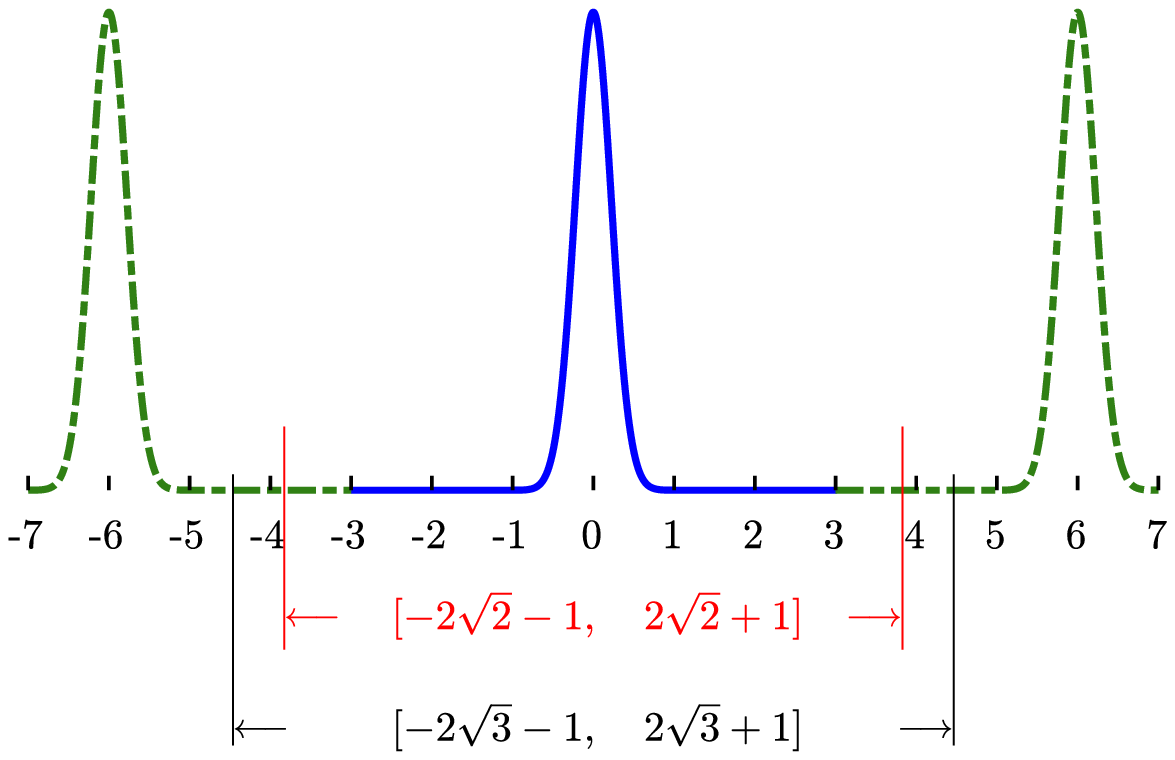}
\caption{Periodic extensions of the twofold (left) and threefold (right) zero-padding for the density. Here $L =1$ for simplicity.}
\label{periodicExtension}
\end{figure}

On domain $\bR_{SL}$, the density $\rho$ is well resolved by the following finite Fourier series
\be
\rho(\bz) \approx \sum_{\bk\in \Lambda} ~\widehat{\rho}_{\bk}~  e^{i \bk \cdot \bz},  \quad \quad  \bz \in \textbf{R}_{S L},
\label{rhoApprox}
\ee
where $\Lambda =\{ \bk:=\frac{ \pi}{S L }(k_1,\dots, k_d) \in \frac{ \pi}{SL} \mathbb Z^d~~\big| k_j= -\frac{SN}{2},\dots, \frac{SN}{2}-1, j = 1,\dots,d\}$ is the mesh grid in Fourier space.
The Fourier coefficients are given as follows
\be
\widehat{\rho}_{\bk} = \frac{1}{(2SL)^{d}} \int_{\textbf{R}_{S L}} \rho(\bz) e^{-i \bk \cdot \bz} { \rm d } \bz,\quad \quad \bk \in \Lambda.
\label{rho_hat}
\ee
Such integral is well approximated by applying the trapezoidal rule, and the resulted summation, which is given explicitly as follows:
\bea\label{rhoTildeK}
\widetilde \rho_\bk  = \frac{1}{(S N)^d} \sum_{\bz_\bp\in \mathcal T_{SN}}
\rho(\bz_\bp) e^{-i\bz_\bp\cdot \bk}, \quad ~\bk \in \Lambda,
\eea
is accelerated by discrete Fast Fourier Transform(FFT) within ${\mathcal O}( d (SN)^d \log(SN) )$ float operations\cite{ShenBook,NickBook}.
Then we obtain the following finite Fourier series
 \be \label{DisFourierSeries}
\rho_{N}(\bz) := \sum_{\bk\in \Lambda} \widetilde{\rho}_{\bk} ~  e^{i \bk \cdot \bz},  \quad ~\bz \in \textbf{R}_{S L},
\ee which is an interpolation on $\mathcal T_{SN}$ and a spectral approximation of $\rho(\bx)$ over $\bR_{SL}$.
As is shown earlier, the periodic extension of $\rho_N(\bz)$ is also a spectral approximation over $\bR_{(2\sqrt{d}+1)L}$, therefore, after
substituting $\rho_N$ for $\rho$ in \eqref{KTM}, we obtain
\begin{eqnarray}
\Phi(\bx)&\approx&\int_{\mathbf{B}_G} U (\by) \rho_{N}(\bx- \by) { \rm d} \by, \nonumber  \\
&= & \sum_{\bk\in \Lambda} \left(\int_{\mathbf{B}_G}~ U(\by) e^{-i \bk \cdot \by} { \rm d} \by \right) ~\widetilde{\rho}_{\bk} ~e^{i \bk \cdot \bx} , \notag \\
&:=& \sum_{\bk\in \Lambda} ~\widehat{U_G}(\bk) ~\widetilde{\rho}_{\bk} ~e^{i \bk \cdot \bx} := \Phi_{N}(\bx), \quad \bx \in \bR_L \label{FormAlg},
   \label{phi_num}
\end{eqnarray}
where $\widehat{U_G}(\bk)$, Fourier transform of the truncated kernel, is defined as
\begin{equation}
   \widehat{U_G}(\bk):= \int_{\mathbf{B}_G} U(\by) ~e^{-i \bk \cdot \by} { \rm d } \by.
   \label{FourTransTrunKer}
\end{equation}
The above approximation of $\Phi$ is spectrally accurate, and its evaluation on target grid $\mathcal T_N$ are accelerated by FFT with great efficiency.

The truncated kernel's Fourier transform is computed once for all and can be treated as a precomputation.
For symmetric kernels, i.e., $U(\bx) = U(r)$ with $r=|\bx|$, the Fourier transform will be reduced to one-dimensional integral as follows\cite{FastConvGreengard}
\begin{equation}
\widehat{U_G}(\bk)=
         \left\{\begin{array}{ll}
               2 \pi \int_{0}^{G} J_0(k r) U(r) r ~{ \rm d } r, & d = 2,\\[0.5em]
               4 \pi \int_{0}^{G} \frac{\sin{(k r)}}{k r}  U(r) r^2~ { \rm d} r,  & d= 3,
         \end{array}\right.
   \label{Fourier-Symmetric}
\end{equation}
where $J_0$ is the Bessel function of first kind with index $0$.
For the most common classical kernels, e.g., the Poisson and Coulomb kernels, Eqn.\eqref{Fourier-Symmetric} has analytical expressions and
we refer to \cite{pcg-dbec,FastConvGreengard} for more details. For general kernels, one may resort to numerical integration, e.g., the Gauss-Kronrod quadrature \cite{pcg-dbec},
or some high-precision library such as the Advanpix toolbox \cite{advanpix}, to obtain accurate evaluation of $\widehat U_G(\bk)$.

Once $\widehat{U_G}(\bk)$ is available, the calculation of $\Phi$ consists of four steps and is summarized in the following algorithm.
\begin{algorithm}
\caption{Fast computation of $\Phi$ on $\mathcal T_N$ in Eqn.\eqref{FormAlg} by Kernel Truncation Method.}
\label{alg:ktm:descri}
\begin{algorithmic}[1]
\REQUIRE Precompute the Fourier transform of the truncated kernel $\widehat{U_G}(\bk)$.
\STATE Extend the density $\rho(\bx)$ to $\bR_{SL}$ by zero-padding.\\[0.3em]
\STATE Compute $\widetilde \rho_\bk$ in Eqn.\eqref{rhoTildeK} via FFT.\\[0.3em]
\STATE Compute $\widetilde \rho_\bk \widehat{U_G}(\bk)$ by pointwise multiplication.
\STATE Compute $\Phi$ in Eqn.\eqref{FormAlg} on mesh grid $\mathcal T_N$ via iFFT.
\end{algorithmic}
\end{algorithm}

\subsection{Anisotropic density }\label{AnisotropicCase}
In this subsection, we study the optimal zero-padding factor for the anisotropic density case,
a situation that is frequently encountered in lower-dimensional confined quantum systems \cite{GS_dynamic, DimRdct}.
We assume that the density is compactly supported in an anisotropic rectangle
\bea
\textbf{R}_{L}^{\bm{\gamma}}:= \prod_{j=1}^d  [-L \gamma_j, L \gamma_j ], \eea
which is also domain of interest, with the anisotropy vector $\bm{\gamma}=(\gamma_1, \dots, \gamma_d)\in \mathbb R^d$.
Without loss of generality,
we take $\gamma_1=1$ and $0<\gamma_j \le 1$ for $j=2, \dots, d$ and define the anisotropy strength as $$\gamma_f := \prod_{j=1}^d \gamma_j^{-1}.$$
A smaller $\gamma_j$ implies a shorter extend of the density $\rho(\bx)$ in the $j$-th direction, and leads to a stronger anisotropy strength $\gamma_f$.
The density is sampled on a uniform mesh grid with the same number of grid points in each spatial direction (denoted by $N$).
Similar as \eqref{KTM}, we have
\bea
\Phi(\bx)= \int_{\mathbf{B}_G} U (\by) \rho(\bx- \by) { \rm d} \by  , \quad \quad \bx \in \textbf{R}_{L}^{\bm{\gamma}},
\label{KTM_anis}
\eea
where the radius $G = 2 L\sqrt{1+\gamma_2^2+ \cdots \gamma_d^2}$ is also the diameter of $\bR_{L}^{\bm \gamma}$.
To integrate \eqref{KTM_anis}, we need to approximate the density on a large domain, because for any $\bx \in\bR_{L}^{\bm \gamma}$ and $\by \in \mathbf B_G$ we have
\begin{equation}
   \bx-\by \in \widetilde \bR := \{\bz \in \mathbb R^{d} ~\big| ~|z_j|\le G + L\gamma_j, j = 1,\ldots, d \} .
\end{equation}
We choose to apply the Fourier spectral method on $\bR_{L\bm{S}}^{\bm \gamma}:= \prod_{j = 1}^d [-S_j L\gamma_j, S_j L\gamma_j] $,
where $\bm{S} = (S_1,\ldots, S_d )\in \mathbb R^d$ with $S_j\geq 1$ being the zero-padding factor in the
$j$-th spatial direction, so to obtain a spectral approximation of the density on $\widetilde \bR$.
Similarly, once the following condition
\begin{equation}
   G + L \gamma_j \le (2 S_j -1) L \gamma_j  \quad \Longrightarrow \quad 1 +  \gamma_j^{-1} \sqrt{1+\gamma_2^2+\ldots+ \gamma_d^{2}} \le S_j,
\end{equation}
is satisfied, we derive the optimal factor as follows
\begin{equation}
   \scalebox{1.2}{\boxed{ \quad \quad  S_j = 1  + \gamma_j^{-1}\sqrt{1+\gamma_2^2+\ldots+ \gamma_d^{2}},~~ j = 1,\ldots, d. }}
\end{equation}
The above factor will be reduced to \eqref{OptZeroPad} if the anisotropic strength is one, i.e., $\gamma_f = 1$.
It is clear that the optimal zero-padding factor may be different along each spatial direction for the anisotropic density case.
In practice, the optimal factor along the $j$-th direction is $\lceil S_j \rceil$,
and the minimum memory cost is $\prod_{j =1}^{d} \lceil S_j\rceil$ times that of the original density.
Take the 2D case for example, when the anisotropy strength $\gamma_f$ gets stronger, we have
\begin{equation}
   S_1 =  1+\sqrt{1+\gamma_2^2} \approx 2 + \frac{\gamma_2^2}{2}, \quad
S_2 =  1+ \gamma_2^{-1}\sqrt{1+\gamma_2^2} \approx \frac{1}{\gamma_2} +1+\frac{\gamma_2}{2},\quad \gamma_2 \rightarrow 0.
\end{equation}
Therefore, the minimum memory cost scales linearly with respect to the anisotropy strength as $2 ( 1+ \gamma_f)$. The stronger anisotropy strength, the more
memory storage it requires.

\

Similarly, after plugging in the density's Fourier series's approximation, we obtain
\bea
\Phi(\bx) \approx \sum_{\bk} ~~\widehat{U_G}(\bk) ~\widehat{\rho}_{\bk} ~e^{i \bk \cdot \bx}, \quad \quad \bx \in \bR_{L}^{\bm \gamma},
\label{phi_num_anis}
\eea
where $\bk = \frac{\pi}{L}~(\frac{k_1}{S_1}, \dots, \frac{k_d}{S_d}), k_j=-S_j N/2, \dots, S_j N/2-1$.
The density's Fourier coefficients
\beas
\widehat{\rho}_{\bk} = \frac{1}{|\bR_{L \bm {S}}^{\bm \gamma}|} \int_{\bR_{L \bm {S}}^{\bm \gamma}} \rho(\bz) e^{-i \bk \cdot \bz} { \rm d} \bz,
\eeas
is well resolved by applying trapezoidal rule on uniform mesh grid, and the resulted summation can be accelerated by FFT.

\subsection{Tensor acceleration}\label{Sect-Tensor}

As pointed out in \cite{FastConvGreengard}, the above algorithm can be rewritten as
a discrete convolution of a tensor and density's grid values.
The tensor $T$ is actually the inverse discrete Fourier transform of vector $\{\widehat {U_G}(\bk), \bk \in \Lambda\}\in \mathbb C^{(SN)^{d}}$.
With a change of index, it is easy to check that the discrete convolution structure holds true for both symmetric and non-symmetric kernels.

To be exact, let us take the 2D case as an example. Define the index set
\begin{equation}
\mathcal{I}_N= \left\{ (n, m) \in \mathbb{Z}^2 | -N/2 \le n \le N/2-1, -N/2 \le m \le N/2-1  \right\}.
\end{equation}
The Fourier transform of the density is approximated as follows
\bea
\widetilde{\rho}(\bk_{pq}) = \frac{1}{(S N)^2} \sum \limits_{(n', m') \in \
\mathcal{I}_{ N}} \rho(x_{n'}, y_{m'})e^{-\frac{2\pi i}{S N}(p n' + q m')}, \quad \bk_{pq} \in \Lambda.
\label{rho_hat_discrete}
\eea
Plugging \eqref{rho_hat_discrete} into Eqn.\eqref{phi_num} and switching the summation order, we know that
the discrete potential $\Phi $ on a uniform grid can be rewritten as the following discrete convolution
\bea
\Phi_{n, m}
& := &  \sum_{(n', m') \in \mathcal{I}_N} T_{n-n', m-m'}~\rho_{n', m'}.
\label{discreteconvolution}
\eea
The tensor $T_{n, m}$ is given explicitly as
\bea
T_{n, m}  =  \frac{1}{ (S N)^2 }  \sum \limits_{(p, q) \in \mathcal{I}_{S N}}
\widehat{U_G} \left(\frac{\pi p}{S L}, \frac{\pi q}{S L}\right)  e^{\frac{2\pi i}{S N} (p n+q m)}, \quad (n,m)\in \mathcal{I}_{2N},
\eea and it can be computed by applying inverse discrete Fourier transform on vector $\{\widehat{U_G}(\bk),~\bk\in \Lambda\}$.

Following the same procedure, we derive the very similar convolution structure for the anisotropic density case. The discrete convolution tensor $T_{n, m}$ reads as follows
\begin{equation}\label{Tensor_Aniso}
   T_{n, m}  =  \frac{1}{ S_1 S_2 N^{2} }  \sum \limits_{p=-S_1 N/2}^{S_1 N/2-1}~ \sum \limits_{q=-S_2 N/2}^{S_2 N/2-1}
   \widehat{U_G} \left(\frac{\pi p }{S_1 L \gamma_1}, \frac{\pi q }{S_2 L \gamma_2}\right)~e^{\frac{2\pi i}{ N} (\frac{p n}{S_1}+\frac{q m}{S_2})},
\end{equation}
and it can be computed by inverse FFT within $\mathcal O(S_1 S_2 N^2 \log( S_1 S_2 N^2))$ float operations.

\

Next, we present details on how to accelerate the discrete convolution via FFT/iFFT by using the 1D example. Extension to 2D case is quite straightforward.
The convolution is given below
\bea
\Phi_k
= \sum_{j=-N/2}^{N/2-1}~ T_{k-j}~f_{j}, \qquad k = - \frac{N}{2},  \dots, \frac{N}{2}-1,
\label{discreteconvolution1D}
\eea
and with a simple index change $k\rightarrow k + (1 + N/2) ,~ j \rightarrow j + (1 + N/2)$, it is reformulated as
\bea
\Phi_k
 =   \sum_{j=1}^{N} T_{k-j}~f_{j}, \qquad k = 1, \dots, N.
\label{discreteconvolution1DEqu}
\eea
For convenience, we set $T_{-N} = 0$ and the summation \eqref{discreteconvolution1DEqu} remains unchanged. Then we have
\begin{equation}\label{FourierT1D}
T_{j} = \frac{1}{2N} \sum_{k=1}^{2 N} \widehat{T}_{k} ~ e^{\frac{i 2 \pi  j (k-1) }{2N}}, \qquad j=-N,\dots, N-1,
\end{equation} with
\beas
 \widehat{T}_{k}=
  \sum_{j=-N}^{N-1} T_{j} ~ e^{-\frac{i 2 \pi  j (k-1) }{2N}}
  =\sum_{j=0}^{N-1} T_{j} ~ e^{-\frac{i2 \pi  j (k-1) }{2N}} + \sum_{j=N}^{2N-1 } T_{j-2N}~ e^{-\frac{i 2 \pi  j (k-1) }{2N}}.
\eeas
In fact, the Fourier transform vector $\widehat T:= [\widehat T_1, \ldots,\widehat T_{2N}]$ is just the discrete Fast Fourier transform of vector
$\widetilde{T} = [T_0, T_1, \dots, T_{N-1}, T_{-N},  \dots, T_{-1} ]$.
Substituting \eqref{FourierT1D}  into \eqref{discreteconvolution1DEqu} and switching the summation order, we have
 \beas
 \Phi_k &=& \sum_{j=1}^{N}  \left[ \frac{1}{2N} \sum_{p=1}^{2N}
 \widehat{T}_{p} ~ e^{\frac{2 \pi i (k-j) (p-1) }{2N}}\right] f_{j}
= \frac{1}{2 N} \sum_{p=1}^{2N}  \widehat{T}_{p} ~ e^{\frac{2 \pi i (k-1) (p-1) }{2N}} \left[\sum_{j=1}^{N} f_{j} ~ e^{-\frac{2 \pi i (j-1) (p-1)  }{2N}}\right] \\
 &:=& \frac{1}{2 N} \sum_{p=1}^{2N}  \widehat{T}_{p} ~\widehat{F}_{p} ~ e^{\frac{2 \pi i (k-1) (p-1) }{2N}}, \qquad  k = 1, \dots, N,
 \eeas
with $$\widehat{F}_{p} :=\sum_{j=1}^{N} f_{j} ~ e^{-\frac{2 \pi i (j-1) (p-1)  }{2N}} =\sum_{j=1}^{2N} F_{j} ~ e^{-\frac{2 \pi i (j-1) (p-1)  }{2N}}, \quad
p=1,\ldots, 2N,$$
being the discrete Fourier transform of $F=[f_1, \dots, f_N, 0, \dots,0] \in \mathbb R^{2N}$.
Obviously, in the last step, $\Phi_k $ can be computed by applying iFFT on $\{\widehat{T}_{p}\widehat{F}_{p}\}_{p=1}^{2N}$ within $\mathcal O(2N \log(2N))$ flops.

To summarize, we present detailed step-by-step algorithm proposed in Algorithm \ref{FastDisCon} using the 1D case.
 \begin{algorithm}
\caption{FFT-acceleration for discrete convolution $\Phi_{k}$\eqref{discreteconvolution1DEqu} }
\label{FastDisCon}
\begin{algorithmic}[1]
  \REQUIRE Set $T_{-N}=0$ and compute $\widehat{T}:=\textrm{FFT}(\widetilde{T}) $ with $\widetilde{T}=[T_0, T_1, \dots, T_{N-1}, T_{-N},  \dots, T_{-1} ]$.\\[0.3em]
\STATE Compute $\widehat{F}:= \textrm{FFT}(F)$ with $F=[f_1, \dots, f_N, 0, \dots,0]\in \mathbb{R}^{2N}$.\\[0.3em]
\STATE Compute $ \widehat{\Phi}:=\widehat{T} \widehat{F}$  by pointwise multiplication.\\[0.3em]
\STATE Compute $\widetilde{\Phi}= \textrm{iFFT}(\widehat{\Phi})$, and set $\Phi = \widetilde{\Phi}(1:N)$.
\end{algorithmic}
\end{algorithm}

The discrete convolution can be implemented with a pair of FFT and inverse FFT for a zero-padded vector of length $(2N)^{d}$.
Such acceleration is purely algebraic because the matrix associated with FFT/iFFT happens to be eigenvectors of the discrete tensor $T$,
and we refer the reader to \cite{atkm} for more details.
Once the convolution tensor is available in the precomputation step, our algorithm involves only FFT/iFFT transform and pointwise multiplication on vectors
of length $(2N)^{d}$, and it still holds true for the anisotropic density case no matter how strong is the anisotropy.
Compared with the no-tensor-accelerated original algorithm, memory cost is reduced dramatically from $(\prod_{j=1}^d S_j) N^d$ to $2^d N^{d}$,
and the improvement becomes more substantial as the anisotropy strength gets stronger.

\begin{remark}

To give a vivid illustration of the improvement, let us take a 3D convolution as example.
A typical double-precision computation on a $256^3$ grid requires $3.4$ Gb memory with the optimal threefold zero-padding,
while the fourfold original algorithm requires around $8$ Gb and the reduction factor is $\frac{37}{64}\approx 60\%$.
With tensor acceleration, the effective memory, once the tensor is available,
requires $1$ Gb which is one eighth of that in the fourfold version. For anisotropic density with $\bm \gamma = (1,1,\gamma)$,
the maximum memory cost is around $\lceil 1 + \sqrt{2}/\gamma\rceil /3 \times 3.4$ Gb,
which is about  $15$ Gb when $\gamma = 1/8$, and the computational costs are reduced substantially with tensor acceleration.
\end{remark}

\subsection{Error estimates}\label{Sec:ErrEst}
In this subsection, we shall present the error analysis for the nonlocal potential  whose error originally comes from the approximation of the density function.
To quantify the error estimates, here we define the following norms
 \beas
 \|\Phi -\Phi_{N} \|_{\infty,\bR_L}&:=& \|\Phi -\Phi_{N} \|_{L^{\infty}(\bR_L)} = \sup_{\bx \in  \textbf{R}_{L}} |(\Phi-\Phi_{N})(\bx)|,  \\
  \|\Phi -\Phi_{N} \|_{2,\bR_L}&:=& \|\Phi -\Phi_{N} \|_{L^{2}(\bR_L)} = \left(\int_{\textbf{R}_{L}} |(\Phi-\Phi_{N})(\bx)|^2 {\rm d} \bx \right)^{1/2},
 \eeas
  where $\Phi_{N}(\bx)$ is the Fourier spectral approximation of function $\Phi(\bx)$  and  $N$ is the number of grids in each spatial dimension. We define the semi-norm as follows
    \beas
  |\rho|_{m,\textbf{R}_{L}} := \left( \sum_{ | \bm{\alpha}|=m} \|\partial^{\bm{\alpha}} \rho \|^2_{2, \bm{\mathrm{R}}_{L} } \right)^{1/2},
  \eeas
with $\bm{\alpha}=(\alpha_1, \dots, \alpha_d) \in {\mathbb Z}^d$, $|\bm{\alpha}| = \sum\limits_{i=1}^{d}  \alpha_i$ and $\partial^{\bm{\alpha}}= \partial^{\alpha_1}_{x_1} \cdots \partial^{\alpha_d}_{x_d}$.
We use $A \lesssim B$ to denote $A\leq c B$ where the constant $c>0$ is independent of the grid number $N$.

 \begin{thm} \label{FourierApproximation}
For smooth and compactly supported function $\rho(\bx)$, assuming that $\supp \{\rho\} \subsetneq \bm{\mathrm{R}}_L$, then the following estimates
  \beas
  \|\partial^{ \bm{\alpha}}  \Phi -(\partial^{ \bm{\alpha}}\Phi)_{N} \|_{\infty, \bm{\mathrm{R}}_{L}}  &\lesssim &  N^{-(m-\frac{d}{2}- |\bm{\alpha}|)}~|\rho|_{m,\bm{\mathrm{R}}_{L}} ,\\
  \|  \partial^{ \bm{\alpha}}  \Phi -(\partial^{ \bm{\alpha}}\Phi)_{N} \|_{2, \bm{\mathrm{R}}_{L}} & \lesssim& N^{-(m- |\bm{\alpha}|)} |\rho|_{m,\bm{\mathrm{R}}_{L}},
  \eeas
hold true for any positive integer $m > \frac{d}{2}$, where $(\partial^{\bm \alpha}\Phi)_N = \left[U\ast (\partial^{\bm \alpha} \rho_N)\right] $ denotes the numerical
approximation of $\partial^{\bm \alpha}\Phi$.
specially, when $\bm{\alpha}= \textbf{0}$, we have
  \beas
  \| \Phi -\Phi_{N} \|_{\infty, \bm{\mathrm{R}}_{L}}  \lesssim  N^{-(m-\frac{d}{2})} |\rho|_{m,\bm{\mathrm{R}}_{L}} ,~~\qquad
  \|   \Phi -\Phi_{N} \|_{2, \bm{\mathrm{R}}_{L}}  \lesssim N^{-m} |\rho|_{m,\bm{\mathrm{R}}_{L}}.
  \eeas
\end{thm}

\begin{proof}
First, we shall derive the error estimates for the density approximation.
The interpolation density  $\rho_{N}(\bx)$ is  $2SL$-periodic in each spatial direction, i.e.,
$$\rho_{N}(\bx)=\rho_{N}(\bx+2SL\textbf{p}),\quad   \forall ~\bx \in \mathbb R^d \mbox{  and  } \textbf{p} \in \mathbb Z^d,$$
where $S = \sqrt{d} + 1$ is the optimal zero-padding factor. Using periodicity, it is easy to obtain
\beas\label{inf_norm}
&&\|  \rho_{N} \|_{\infty, \textbf{R}_{(2 \sqrt{d}+1)L} \setminus \textbf{R}_{S L} }=\|  \rho_{N} \|_{\infty, \textbf{R}_{SL} \setminus \textbf{R}_{L} }, \\[0.2em]
&&\label{2_norm}
\| \rho_{N} \|_{2, \textbf{R}_{(2 \sqrt{d}+1)L} \setminus \textbf{R}_{S L} }\lesssim \| \rho_{N} \|_{2, \textbf{R}_{SL} \setminus \textbf{R}_{L} }.
\eeas
Since $\supp\{\rho\} \subsetneq \textbf{R}_L$, we then derive the following estimates
\beas
\| \rho- \rho_{N} \|_{\infty,\textbf{R}_{(2 \sqrt{d}+1)L}} &=& \max \{\| \rho- \rho_{N} \|_{\infty,\textbf{R}_{SL}}, ~ \| \rho- \rho_{N} \|_{\infty,\textbf{R}_{SL} \setminus \bR_{L}}  \}= \| \rho- \rho_{N} \|_{\infty,\textbf{R}_{SL}}. \\
\| \rho- \rho_{N} \|_{2,\textbf{R}_{(2 \sqrt{d}+1)L}} &\lesssim &\| \rho- \rho_{N} \|_{2,\textbf{R}_{SL}}+\| \rho- \rho_{N} \|_{2,\textbf{R}_{SL} \setminus \bR_{L}} \lesssim  \| \rho- \rho_{N} \|_{2,\textbf{R}_{SL}}.
\eeas

Next, we will investigate the potential case ($ \bm{\alpha}= \textbf{0}$) and define the error function as
$$e_{N}(\bx): = (\Phi -\Phi_{N})(\bx),\quad \bx \in \textbf{R}_{L}.$$
To estimate the errors, we have
\beas
|e_{N}(\bx)|& =&     \left| \int_{\mathbf{B}_G} U (\by) \left(\rho-\rho_{N} \right) (\bx-\by) { \rm d} \by \right|,\\
&\le &\max_{\substack{\bx \in \textbf{R}_L \\ \by \in \mathbf{B}_G }} |(\rho - \rho_{N}) (\bx -\by)|  \int_{\mathbf{B}_G} U (\by) { \rm d} \by, \\
&\le&   \| \rho- \rho_{N} \|_{\infty,\textbf{R}_{(2 \sqrt{d}+1)L}} \int_{\mathbf{B}_G} U (\by) { \rm d} \by, \\
&=& \| \rho- \rho_{N} \|_{\infty,\textbf{R}_{SL}} \int_{\mathbf{B}_G} U (\by) { \rm d} \by,\quad \quad \bx \in \textbf{R}_{L}.
\eeas
Using Lemma \ref{FouApprox}(see  \ref{AppFourier}), taking supremum with respect to $\bx$, we obtain
\beas
\|  e_{N} \|_{\infty, \textbf{R}_{L}}  & \lesssim & \| \rho- \rho_{N} \|_{\infty,\textbf{R}_{SL}}  \lesssim  N^{-(m-\frac{d}{2})} |\rho|_{m,\textbf{R}_{SL}}= N^{-(m-\frac{d}{2})} |\rho|_{m,\textbf{R}_{L}}, \forall ~m > d/2.
\eeas
As for the $L^2$ norm, using Cauchy-Schwarz inequality, we obtain
\beas
|  e_{N}(\bx) |^2 & = &    \left|\int_{\mathbf{B}_G} U (\by) \left(\rho-\rho_{N} \right) (\bx-\by) { \rm d} \by \right|^2
  \le  \left( \int_{\mathbf{B}_G} U (\by) { \rm d} \by \right) \left( \int_{\mathbf{B}_G} U (\by) \left(\rho-\rho_{N} \right)^2 (\bx-\by) { \rm d} \by \right),
\eeas
Integrating the above equation on both sides with respect to $\bx$ over $\textbf{R}_L$ and utilizing Lemma \ref{FouApprox},
we have
\beas
\| e_{N} \|_{2, \textbf{R}_{L}}  & \le &  \left( \int_{\mathbf{B}_G} U (\by) { \rm d} \by \right)^{1/2} \left(  \int_{\textbf{R}_L} \int_{\mathbf{B}_G} U (\by) \left(\rho-\rho_{N} \right)^2 (\bx-\by) { \rm d} \by { \rm d} \bx \right)^{1/2}  \\
 & \le &   \| \rho- \rho_{N} \|_{2,\textbf{R}_{(2\sqrt{d}+1)L}}\left( \int_{\mathbf{B}_G} U (\by) { \rm d} \by \right)
  \\
  &\lesssim&   \| \rho- \rho_{N} \|_{2,\textbf{R}_{SL}} \lesssim N^{-m} |\rho|_{m,\textbf{R}_{SL}}= N^{-m} |\rho|_{m,\textbf{R}_{L}}.
\eeas
Lastly, for the general case $\bm{\alpha} \not = \textbf{0}$, interchanging the order of convolution and differentiation, we know that
\beas
\partial^{ \bm{\alpha}} \Phi(\bx) =\partial^{ \bm{\alpha}} ( U \ast \rho )(\bx) =( U \ast \partial^{ \bm{\alpha}} \rho)(\bx).
\eeas
Similarly, by applying results of the case $\bm{\alpha}= \textbf{0}$ by substituting $\partial^{\bm \alpha} \rho$ for $\rho$, and combining Lemma \ref{FouApprox}, we can complete the proof.
\end{proof}

\begin{remark}
For smooth and fast-decaying density,
it is reasonable to truncate the whole space into bounded domain such that the truncation errors are negligible, therefore,
we can treat such density as smooth and compactly supported function and apply Theorem \ref{FourierApproximation}.
In this case, the Fourier approximation is of spectral accuracy.
\end{remark}

\begin{remark}
For fast-decaying but non-smooth density, the convergence rates of the potential and its derivatives are limited by the density's regularity,
and we refer readers to \ref{AppFourier} for more details on the Fourier approximation of low-regularity function.
\end{remark}

\section{Numerical results} \label{NumericalResults}

\setcounter{equation}{0}
In this section, we shall investigate the effect of optimal zero-padding factor for the nonlocal potential evaluation with
isotropic and anisotropic densities, on the accuracy and efficiency for different kernels and smooth/nonsmooth density.
The computational domain $\bR_{L}^{\bm \gamma}$, where $\bm{\gamma} = (\gamma_1,\ldots, \gamma_d)$ is the anisotropy vector, is discretized uniformly in each spatial direction with mesh size $h_j$,
and we define  mesh size vector as $\bh = (h_1,\ldots, h_d)$.
For simplicity, we shall use $h$ and $\bm{S} = (S_1,\ldots, S_d)$ to denote the mesh size if all the mesh sizes are equal, and the zero-padding vector respectively.
The algorithms were implemented in Matlab (2016a) and run on a 3.00GH Intel(R) Xeon(R) Gold 6248R CPU with a 36 MB cache in Ubuntu GNU/Linux.

All the numerical errors presented here are measured in relative maximum norm, i.e.,
\begin{equation}
   {\mathcal E}:=\frac{\|\Phi-\Phi_{\bh}\|_{l^\infty}}{\|\Phi\|_{ l^\infty}}
=\frac{\textrm{max}_{\bx \in \mathcal T_{\bh}} |\Phi(\bx)-\Phi_{\bh}(\bx)|}{\textrm{max}_{\bx \in \mathcal T_{\bh}} |\Phi(\bx)|}, %
\end{equation}
where $\Phi_{\bh}$ is the numerical solution on mesh grid $\mathcal T_{\bh}$ and $\Phi(\bx)$ is the exact solution.

\begin{exmp} \label{3DPoissonInteraction}
The 3D Poisson potential with $U(\bx)=\frac{1}{4 \pi |\bx|}$. For anisotropic source $\rho_0( \bx ) =e^{-(x^2+y^2+z^2/\gamma_3^2)/\sigma^2}$ with $\sigma>0$ and $\gamma_3 \le 1$,
the nonlocal potential is given analytically as \cite{GauSum}
\be
\Phi_0(\bx)=\left\{\begin{array}{l}
\frac{\sigma^3 \sqrt{\pi}}{4 |\bx|} \textrm{Erf} ~ (\frac{|\bx|}{\sigma}), ~~~~~~~~~~~~~\quad~~~~~~ \gamma_3=1,\\
\\
\frac{\gamma_3 \sigma^2}{4} \int_0^{\infty} \frac{e^{-\frac{x^2+y^2}{\sigma^2(t+1)}} e^{-\frac{z^2}{\sigma^2 (t+ \gamma_3^2)}}}{(t+1) \sqrt{t+\gamma_3^2}} { \rm d} t, ~~~ \gamma_3 \not=1,\\
\end{array}\right.
\ee
where $ \textrm{Erf} ~(x)=\frac{2}{\sqrt{\pi}} \int _0 ^x e^{-t^2} { \rm d} t$ is the  error function \cite{handbook}.
We consider the following cases
\begin{itemize}
\item
\textbf{Case I:} The anisotropic source $\rho(\bx)=\rho_0(\bx)$ with $\Phi(\bx)=\Phi_0(\bx)$.
\item
\textbf{Case II:} Shifted source $\rho(\bx)=\rho_0(\bx)+ \rho_0(\bx-\bx_0)$ with $\Phi(\bx)=\Phi_0(\bx)+ \Phi_0(\bx-\bx_0)$.
\end{itemize}

\end{exmp}
Table \ref{Coulomb3Dtkm} presents errors of the 3D Poisson potential for isotropic density of \textbf{Case I}, i.e.,$\gamma_3 = 1$, in the upper part with $L = 8,\sigma = \sqrt{1.2}$.
Errors of the potential generated by anisotropic densities are shown in the lower part: \textbf{Case I} is computed with  $ L=12,\sigma=2, \bh=\frac{1}{2}\bm{\gamma}$,
and \textbf{Case II} is calculated with $L=16,\sigma=2, \bx_0=(2,2,0)$ and $\bh=\frac{1}{4}\bm{\gamma}$.
The optimal zero-padding factor is chosen as the integer multiple of $\frac{1}{2}$, e.g.,$ S = 2.5 \approx  \sqrt{2} + 1$.
From Table \ref{Coulomb3Dtkm}, we can conclude that for the 3D nonlocal potential computation via KTM:
(i) the twofold zero-padding is insufficient, the fourfold zero-padding is redundant and the threefold is optimal.
(ii)  the optimal factor for the anisotropic density grows linearly with the anisotropy strength $\gamma_f$.

\

To investigate the efficiency, we present the performance of KTM with/without tensor acceleration in terms of memory costs
and CPU time in Table \ref{CompareKTM}, where we choose isotropic density in \textbf{Case I} with $L = 8,\sigma = \sqrt{1.2}$ and $h=1/16$.
The computation is split into two parts: the precomputation part(\textbf{PreComp}): computation of the tensor $\widehat{T}_{\bk}$ or $\widehat{U_G}$ in the tensor/no-tensor version,
and the execution part (\textbf{Execution}).
In Figure \ref{TimeTestKTM3D},  we plot the CPU time of \textbf{PreComp} and \textbf{Execution} with/without tensor acceleration version for different anisotropy strength $\gamma_f$.

From Table \ref{CompareKTM}, we can see that, compared with the classical fourfold zero-padding version, KTM with optimal factor greatly reduces the memory and computational costs,
therefore, computation of a larger size is now possible on the personal computers.
Moreover, once the precomputation step is done, the effective computation consists of only a pair of FFT/iFFT
on  vectors of twice the length in each direction, no matter how strong anisotropy strength gets,
and this facilitate multiple potential evaluations on the same setups.
It is worth to note that the seemingly incorrect drop down at $\gamma_f = 2$ in the ``Tensor-PreComp'' curve is actually
caused by the integer multiple of $0.5$ selection in $S_i$.  The total zero-padding factor $\prod_{i=1}^{3}S_i$ is $25$ for $\gamma_f =2$ and $27$ when $\gamma_f =1$ (the isotropic case).

\begin{table}[!htbp]
\centering
\caption{Errors of  the 3D Poisson potential for isotropic (upper part) and anisotropic (lower part) density in \textbf{Example }\ref{3DPoissonInteraction}.}
\label{Coulomb3Dtkm}
\begin{tabular}{ccccc}
\toprule
 $h$& $ 2 $ & $1$ & $ 1/2 $ & $ 1/4 $ \\
  \midrule
$S=2$& 4.2908E-01&  1.0283E-01 & 1.0276E-01 &  1.0276E-01  \\
$S = 3$&  4.2117E-01&  2.9848E-03 & 1.8552E-08 &  3.7007E-16   \\
$S = 4$& 4.2046E-01&  2.9596E-03 & 2.0106E-08 &  3.7007E-16  \\
\midrule
  $\gamma_3$& $ 1 $ & $1/2$ & $ 1/4 $ & $ 1/8 $  \\
  \midrule
\textbf{Case I}  & 3.3307E-16&  5.4171E-15 & 4.8932E-15 &  3.8102E-15  \\
\textbf{Case II}  & 5.1902E-16&  5.6243E-15 & 5.3014E-15 &  4.1688E-15 \\
$\bm {S}$& (3, 3, 3)&  (2.5, 2.5, 4) & (2.5, 2.5, 7) & (2.5, 2.5, 12.5)\\
\bottomrule
\end{tabular}
\end{table}

\begin{table}[!htbp]
\centering
\caption{The performance of KTM for 3D Poisson potential with isotropic density in \textbf{Case I} of \textbf{Example }\ref{3DPoissonInteraction}.}
\label{CompareKTM}
\setlength{\tabcolsep}{4mm}{
\begin{tabular}{cccccc}
\toprule
\multirow{2}{*}{Tensor}  & \multirow{2}{*}{$S$} &\multicolumn{2}{c} {\textbf{PreComp}}  &\multicolumn{2}{c} {\textbf{Execution}}\\
\cline{3-6}
& &Mem(Gb) &Time(s) & Mem(Gb)&Time(s) \\
\midrule
\multirow{2}{*}{Yes} &$3$   &4.4 & 59.80 &2.3& 19.31 \\
          &$4$  & 9.0 & 158.58&2.3&   19.32     \\
\midrule
\multirow{2}{*}{No}    &$3$    &3.4 & 12.53 &7.0& 61.18\\
      &$4$     & 8.0 & 32.96&16.3&    167.00  \\
\bottomrule
\end{tabular} }
\end{table}

\begin{figure}
\centering
\includegraphics[width=0.6\textwidth]{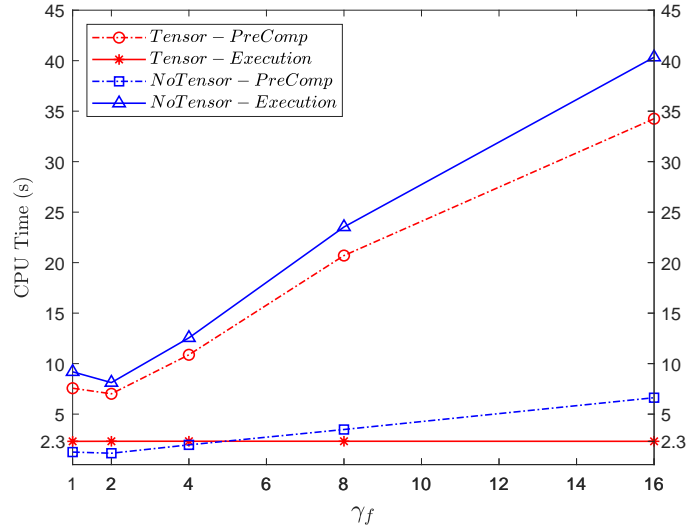}
\caption{Timing results of KTM versus increasing anisotropy strength $\gamma_f$. }
\label{TimeTestKTM3D}
\end{figure}

\begin{exmp}\label{exmp:Poisson}
   The 1D and 2D Poisson potential with convolution kernel
   \begin{equation*}
      U(\bx) =  \left\{
         \begin{array}{ll}
            -\frac{1}{2} |x|,     & \mbox{1D Poisson},\\
            -\frac{1}{2\pi} \ln(|\bx|),   & \mbox{2D Poisson}.
         \end{array}
      \right.
   \end{equation*}
First, we consider an isotropic Gaussian source density  $\rho(\bx) = e^{-|\bx|^2/ \sigma ^2}$. The 1D and 2D Poisson potential reads as \cite{atkm}
   \begin{equation*}
      \Phi(\bx) =  \left\{
         \begin{array}{ll}
            -\frac{\sigma^2}{2} e^{-x^2/\sigma^2}-\frac{\sqrt{\pi} \sigma}{2} x \textrm{Erf}~(\frac{x}{\sigma}),     & \mbox{1D Poisson},\\ [0.4em]
 =-\frac{\sigma^2}{4} \left[ E_1(\frac{|\bx|^2}{\sigma^2}) + 2 \ln(|\bx|) \right],   & \mbox{2D Poisson},
         \end{array}
      \right.
   \end{equation*}
where $E_1(r):= \int_r^{\infty} t^{-1} e^{-t} dt $  for $ r>0 $ is the exponential integral function \cite{handbook}.\\
\noindent Next, we consider an anisotropic density $\rho(\bx)$, which is the Laplacian of $\Phi(\bx)=e^{-\frac{x^2}{\sigma^2}-\frac{y^2}{\alpha^2}}$, i.e.,
   \beas
    \rho( \bx ) = -\Delta \Phi(\bx) = \Phi(\bx) \left(-\frac{4 x^2}{\sigma^4} - \frac{4 y^2}{\alpha^4}+ \frac{2}{\alpha^2} +\frac{2}{\sigma^2} \right).
    \eeas
\end{exmp}

Table \ref{tab:Poisson} presents errors of the 1D (upper) and 2D (lower) Poisson potentials
that are computed with $L=8$, $\sigma=\sqrt{1.2}$, from which
we can see clearly that twofold/threefold zero-padding is optimal for the 1D and 2D problem respectively.
Moreover, if a fractional zero-padding is considered,
\textbf{2.5}-fold padding is the most efficient for 2D problem in terms of memory cost and computation time.
Table \ref{Poisson2Dtkm_anis} presents errors of Poisson potential with anisotropic density
with  $L=10,\sigma=1.2, \alpha=\gamma_2 \sigma, \bm{\gamma} = (1,\gamma_2)$ and $\textbf{h}=\frac{1}{4}\bm{\gamma}$.

\begin{table}[!htbp]
\centering
\caption{Errors of the 1D (upper) and 2D (lower) Poisson potentials for isotropic density in \textbf{Example}\ref{exmp:Poisson}.}
\label{tab:Poisson}
\begin{tabular}{rcccc}
\toprule
$ h $  & $ 2 $     & $1$         & $ 1/2 $    & $ 1/4 $ \\
\hline
  $S = 1$& 2.1299&  1.9228 & 1.9227 &  1.9227   \\
  $S = 2$& 1.0356E-01&  1.3938E-04 & 6.3941E-10 &  4.5744E-16   \\
  $S = 3$& 1.1294E-01&  2.2359E-03 & 1.3244E-07 &  6.4328E-16   \\
\midrule
$S = 2  $& 3.0583E-01&  4.5718E-02 & 4.7154E-02 &  4.7154E-02 \\
$S = 2.5$& 2.0928E-01&  1.4561E-03 & 4.8882E-08 &  1.6780E-15 \\
$S = 3  $& 2.2027E-01&  1.0051E-03 & 2.0456E-08 &  1.6780E-15   \\
$S = 4  $& 2.2778E-01&  1.3082E-03 & 7.9905E-08 &  1.1441E-15 \\
\bottomrule
\end{tabular}
\end{table}

\begin{table}[!htbp]
\centering
\caption{Errors  of the 2D Poisson potential for anisotropic density in \textbf{Example} \ref{exmp:Poisson}.}
\label{Poisson2Dtkm_anis}
\begin{tabular}{rccccc}
  \toprule
   $\gamma_2$& $ 1 $ & $1/2$ & $ 1/4 $ & $ 1/8 $ & $1/16$ \\
  \midrule
$\mathcal E$& 5.1070E-15&  5.3429E-15 & 1.0596E-14 &  3.5612E-14 & 3.1667E-14 \\
   \midrule
 $\bm {S}$ & (2.5, 2.5)&  (2.5, 3.5) & (2.5, 5.5) & (2, 9.5) & (2, 17.5) \\
\bottomrule
\end{tabular}
\end{table}

\begin{exmp}\label{exmp:2d:Coulomb}
The 2D Coulomb potential with $U(\bx)=\frac{1}{2 \pi |\bx|}$.
For anisotropic source $\rho( \bx ) =e^{-\frac{1}{\sigma^2} (x^2+\frac{y^2}{\gamma_2^2})}$ with $\gamma_2 \le 1$,
the 2D Coulomb potential is given analytically as \cite{GauSum}
  \be
\Phi(\bx)=\left\{\begin{array}{l}
\frac{\sqrt{\pi} \sigma}{2} I_0(\frac{|\bx|^2}{2 \sigma^2}) e^{-\frac{|\bx|^2}{2 \sigma^2}}, ~~~~~~~~~~\quad~~~~~~ \gamma_2=1,\\
\\
\frac{\gamma_2 \sigma}{\sqrt{\pi}} \int_0^{\infty} \frac{e^{-\frac{x^2}{\sigma^2(t^2+1)}} e^{-\frac{y^2}{\sigma^2 (t^2+ \gamma_2^2)}}}{\sqrt{t^2+1} \sqrt{t^2+\gamma_2^2}} { \rm d} t, ~~~~ \gamma_2\not=1,\\
\end{array}\right.
\ee
where $I_0(x)$ is the modified Bessel function of the first kind \cite{handbook}.
\end{exmp}

Table \ref{Coulomb2Dtkm} shows errors of the 2D Coulomb potential $\Phi$ and its derivative $\partial_x \Phi$ for isotropic density, i.e., $\gamma_2 =1$,
$L=8,\sigma=\sqrt{1.2}$, from which we can see clearly that the potential and its derivative can both be computed with spectral accuracy.
Table \ref{Coulomb2Dtkm_anis} shows errors of the potential with anisotropic density with $L=12, \sigma=1.5$ and $\textbf{h}=\frac{1}{4}(1, \gamma_2)$.

\begin{table}[!htbp]
\centering
\caption{Errors  of the 2D Coulomb potential $\Phi$ and $\partial_x \Phi$ for isotropic density in \textbf{Example} \ref{exmp:2d:Coulomb}.  }
\label{Coulomb2Dtkm}
\begin{tabular}{rccccc}
\toprule
   $\Phi$& $ h = 2 $ & $h = 1$ & $ h = 1/2 $ & $ h =  1/4 $   \\
  \midrule
$S = 2$  & 2.0661E-01&  2.3210E-03 & 1.0203E-03 &  1.0203E-03 \\
$S = 2.5$& 2.0771E-01&  2.3577E-03 & 2.6029E-08 &  4.5744E-16\\
$S = 3$  &2.0822E-01&  2.2739E-03 & 2.7514E-08 &  5.7180E-16 \\
$S = 4$  &2.0869E-01&  2.3884E-03 & 2.8837E-08 &  3.4308E-16  \\
   \midrule
   $\partial_x\Phi$& $ h = 2 $ & $h = 1$ & $ h = 1/2 $ & $ h =  1/4 $   \\
   \midrule
$S = 2$   & 1.0794&  2.8659E-02 & 6.3223E-03 &  6.3224E-03        \\
$S = 2.5$ & 1.0789&  2.8563E-02 & 1.7366E-06 &  8.7677E-16        \\
$S = 3$   & 1.0790&  2.8550E-02 & 1.7343E-06 &  6.8106E-16        \\
$S = 4$   & 1.0791&  2.8540E-02 & 1.7330E-06 &  7.3194E-16        \\
\bottomrule
\end{tabular}
\end{table}

\begin{table}[!htbp]
\centering
\caption{Errors  of the  2D Coulomb potential for anisotropic densities in \textbf{Example} \ref{exmp:2d:Coulomb}.}
\label{Coulomb2Dtkm_anis}
\begin{tabular}{rccccc}
\toprule
$\gamma_2$& $ 1 $ & $1/2$ & $ 1/4 $ & $ 1/8 $ & $1/16$ \\
\midrule
${\mathcal E}$ & 5.8462E-16&  2.3117E-15 & 1.4986E-15 &  1.6609E-15 & 2.0713E-15 \\
 \midrule
$\bm{S}$& (2.5, 2.5)&  (2.5, 3.5) & (2.5, 5.5) & (2, 9.5) & (2, 17.5)  \\
\bottomrule
\end{tabular}
\end{table}

\begin{exmp}\label{exmp:DDI} The dipole-dipole interaction(DDI).\\
\noindent \textbf{The 3D DDI.} The kernel is given as
\beas
 U(\bx)& = & \frac{3}{4 \pi} \frac{\textbf{m} \cdot \textbf{n} - 3 (\bx \cdot \textbf{m}) (\bx \cdot \textbf{n})/|\bx|^2 }{|\bx|^3} ,~~~ \bx \in \mathbb{R}^3,
  \eeas
  where $\textbf{n}$, $\textbf{m} \in \mathbb{R}^3$ are unit vectors representing the dipole orientations, and
the 3D potential is reformulated as \cite{MathematicalBEC,GS_dynamic,NUFFT_gsdy,NUFFT,TrappedAtomic_anis,TrappedAtomic_dipole}:
  \be
  \Phi(\bx)= -(  \textbf{m} \cdot \textbf{n} ) \rho(\bx)-3 \frac{1}{4 \pi |\bx|} \ast (\partial_{\textbf{n} \textbf{m}} \rho),
\ee
where $\partial_{\textbf{m}} = \textbf{m} \cdot  \nabla $ and $ \partial_{\textbf{n}\textbf{m}}= \partial_{\textbf{n}} (\partial_{\textbf{m}})$.
In fact, the potential can be calculated via the 3D Poisson potential with source term $\partial_{\textbf{n}\textbf{m}} \rho$,
which can be easily computed numerically via Fourier spectral method \cite{NickBook}.

\noindent \textbf{The quasi-2D DDI.} The kernel is given as \cite{Quasi2Dpra,MathematicalBEC}
 \beas
  U(\bx)&=&-\frac{3}{2} ( \partial_{\bn_\perp\bn_\perp}-
n_3^2\nabla^2)\frac{1}{(2\pi)^{3/2}}\int_{\mathbb R}
\frac{e^{-s^2/2}}{\sqrt{|\bx|^2+\varepsilon^2s^2}}{\rm d}s \\
&:=&-\frac{3}{2} ( \partial_{\bn_\perp\bn_\perp}-
n_3^2\nabla^2)  \widetilde{U}(\bx), \quad \bx\in{\mathbb R}^2,
\eeas
where $\varepsilon >0$,$\nabla_\perp^2=\Delta $, $\bn_\perp=(n_1,n_2)^T$,
$\partial_{\bn_\perp}=\bn_\perp\cdot\nabla_\perp$ and $\partial_{\bn_\perp\bn_\perp}=\partial_{\bn_\perp}(\partial_{\bn_\perp})$, and the 2D potential
can be reformulated similarly as a convolution of $\widetilde U$ and an effective density $\widetilde \rho$ as follows
\begin{equation}
   \Phi(\bx) =\widetilde U \ast \left[ -\frac{3}{2} ( \partial_{\bn_\perp\bn_\perp}-
   n_3^2\nabla^2) \right ] \rho  = \widetilde U \ast \widetilde \rho.
\end{equation}
For a Gaussian density $\rho( \bx ) =e^{-|\bx|^2/ \sigma ^2}$, the DDI potentials are given explicitly as
\begin{equation}
\Phi(\bx) =
  \left\{\begin{array}{ll}
  -(  \textbf{m} \cdot \textbf{n} ) \rho(\bx)- 3 \partial_{\textbf{n} \textbf{m}} \left( \frac{\sigma^2 \sqrt{\pi}}{4} \frac{\textrm{Erf}~(r/\sigma)}{r/ \sigma} \right)
  , & d = 3,\\[0.5em]

-\frac{8}{\sigma^4} e^{-\frac{r^2}{\sigma^2}} \pi \int_0^\infty \widetilde{U}(t)t e^{-\frac{t^2}{\sigma^2}} [(-\sigma^2+r^2+t^2)I_0(\frac{2 r t}{\sigma^2})-2 r t I_1(\frac{2 r t}{\sigma^2}) ]{\rm d}t,
      & d= 2, \bn = (0,0,1)^T,
         \end{array}\right.
         \label{PhiDDI}
\end{equation}
where $I_0$ and $I_1 $ are the modified Bessel functions of order $0$ and $1$ respectively.

\end{exmp}

Table \ref{tab:DDI} shows errors of the 3D DDI (the upper part), computed with $L = 8, \sigma=\sqrt{1.2}$
with different dipole orientations $\textbf{n}= (0.82778, 0.41505, -0.37751)^{T},\textbf{m}= (0.3118, 0.9378, -0.15214)^{T}$,
and the quasi-2D DDI (the lower part), computed with $L=12, \sigma=2,\varepsilon=\frac{1}{\sqrt{32}}$ and $\bn = (0,0,1)^T$.

\begin{table}
\centering
\caption{Errors of the 3D DDI (upper part) and the quasi-2d DDI (lower part) in \textbf{Example} \ref{exmp:DDI}.}
\label{tab:DDI}
\begin{tabular}{ccccc}
\toprule
  $h$& $ 2 $ & $1$ & $ 1/2 $ & $ 1/4 $ \\
  \hline
  $S = 2$& 1.5254E+00&  8.3971E-02 & 7.4273E-02 &  7.3380E-02  \\
  $S = 3$& 1.5763E+00&  2.9150E-02 & 8.4761E-07 &  7.0062E-15   \\
  $S=4$& 1.5862E+00&  2.9422E-02 & 8.7784E-07 &  7.2865E-15   \\
  \midrule
   $S = 2$&1.9970E-02 & 3.0216E-03 & 3.0222E-03 & 3.0222E-03 \\
  $ S = 2.5$&1.9866E-02 & 4.5791E-06 & 6.4182E-15 & 5.2386E-15 \\
 $S = 3$& 1.5975E-02 & 3.7262E-06 & 6.4768E-15 & 5.2042E-15 \\
 $S = 4$& 1.5525E-02 & 3.6151E-06 & 6.4828E-15 & 5.2386E-15 \\
\bottomrule
\end{tabular}
\end{table}

\begin{exmp}\label{exmp:3d:QQI}
The 3D quadrupolar potential.
\end{exmp}
The quadrupole-quadrupole interaction was first proposed in \cite{QQI}, which reads as
\bea
U(\bx) =  \frac{1}{r^5} Y_{4}^{0}(\theta) =  \frac{3}{16 \sqrt{\pi}} \frac{1}{r^5} \left(3-30 \cos^2 \theta + 35 \cos^4 \theta \right),
\quad \bx \in \mathbb R^3,
 \label{QQI_U}
\eea
where $\theta$ is the polar coordinate, $Y_{4}^{0}$ is the spherical harmonic function \cite{handbook}, and
the Fourier transform is
\be\label{FourQQI} \widehat{U}(\bk)=\frac{4 \pi}{105} k^2 Y_{4}^{0}(\theta_k), \quad \bk \in \mathbb R^3.
\ee
Since $\widehat{U}(\bk)$ is not smooth, standard Fourier method results in saturated convergence order\cite{quasi2D}.
Using partial wave expansion, we obtain the Fourier transform of the truncated kernel as follows
\bea
\widehat{U_G}(\bk) & = &  \int_{\mathbf{B}_G} \frac{Y_{4}^{0} (\theta )}{r^5} e^{-i \bk \cdot \bx} {\rm d} \bx = 4 \pi Y_{4}^{0}(\theta_k) \int_{0}^G \frac{1}{r^3} j_4( k r) { \rm d} r \notag \\
& = &  4 \pi Y_{4}^{0}(\theta_k) \left[ \frac{k^2}{105}- \frac{(-15+G^2 k^2) \cos(G k)}{G^6 k^4}+ \frac{3(-5+ 2 G^2 k^2) \sin(G k)}{G^7 k^5} \right].
\eea
For a Gaussian density $\rho(\bx)=e^{-|\bx|^2/\sigma^2}$, the quadrupolar potential is given explicitly as
\bea
\Phi(\bx) &=& \frac{1}{(2 \pi)^3} \int_{\mathbb{R}^3} \widehat{U}(\bk) \widehat{\rho}(\bk) e^{i \bk \cdot \bx} {\rm d} \bk  = \frac{2}{105 \pi^3} Y_{4}^{0}(\theta) \int_0^{\infty} k^4 \widehat{\rho}(k) j_4(k r) { \rm d} k \notag \\
 & =&\frac{2 \pi \sigma^3}{105} Y_{4}^{0}(\theta) \left[ -e^{-\frac{r^2}{\sigma^2}} \left( \frac{  8 r^6+28 r^4 \sigma^2 + 70 r^2 \sigma^4 +105 \sigma^6 }{r^4 \sigma^7} \right)+ \textrm{Erf}~(\frac{r}{\sigma}) \frac{105 \sqrt{\pi}}{2 r^5} \right].
\eea
Table \ref{QQI} presents errors of the quadrupolar potential computed with $L=12$, $\sigma=1.5 $, zero-padding factor $S$ and mesh size $h$.

\begin{table}[!htbp]
\centering
\caption{Errors of the 3D quadrupolar potential in \textbf{Example} \ref{exmp:3d:QQI}.}
\label{QQI}
\begin{tabular}{ccccc}
  \toprule
 & $ h = 2 $ & $h = 1$ & $ h =1/2 $ & $ h =  1/4 $\\
  \midrule
  $S=2$& 5.3856E-01&  1.3009E-02 & 7.4579E-06 &  7.4579E-06  \\
  $S=3$& 5.4149E-01& 1.3170E-02  & 4.3450E-10 &  3.1796E-14   \\
  $S=4$& 5.4252E-01&  1.3227E-02 & 4.4095E-10 &  2.2994E-14   \\
\bottomrule
\end{tabular}
\end{table}

\begin{exmp}\label{Exp:nonsmooth}
Nonsmooth density. We consider the 1D/2D/3D Poisson potential and 2D Coulomb potential generated by a nonsmooth density, i.e.,
\begin{equation}
\rho(\bx)=\left\{\begin{array}{ll}
      \left(1-|\bx - \bm{\delta} |^2\right)^m,& |\bx-\bm{\delta}| \le 1,\\ [0.2em]
0, &|\bx-\bm{\delta}| >1, \\
\end{array}\right.
\end{equation} with $m \in \mathbb Z^{+}$ and $\bm{\delta}\in \mathbb R^{d}$ being the shift vector.
It is clear that $\rho\in C^{m-1}(\mathbb R^{d})$, and the corresponding nonlocal potentials are detailed explicitly in \ref{exact}.
\end{exmp}

Figure \ref{PoissonNotSmooth} shows the errors and convergence order of $\partial_x \rho$ and $\partial_x \Phi$ of the Poisson potentials,
where $N$ is the grid number in each direction and $\bm{\delta}=\textbf{0}$, $L=2$.
Figure \ref{CoulombNotSmooth} shows the errors  and convergence order  of  $\partial_x \Phi $ generated by the 2D Coulomb kernel with  $L=2$.

From Figure \ref{PoissonNotSmooth}- \ref{CoulombNotSmooth}, we can see that the convergence order for $\partial_x \rho$ is $m-1$,
while it is $m+1$ for $\partial_x \Phi$ in 1D/2D/3D Poisson problem, and is $m$ for the 2D Coulomb case.
The convergence order of the potential is usually higher than that of the density, and it is compatible with Theorem \ref{FourierApproximation} but performs better.

\begin{figure}
\centering
\includegraphics[scale=0.50]{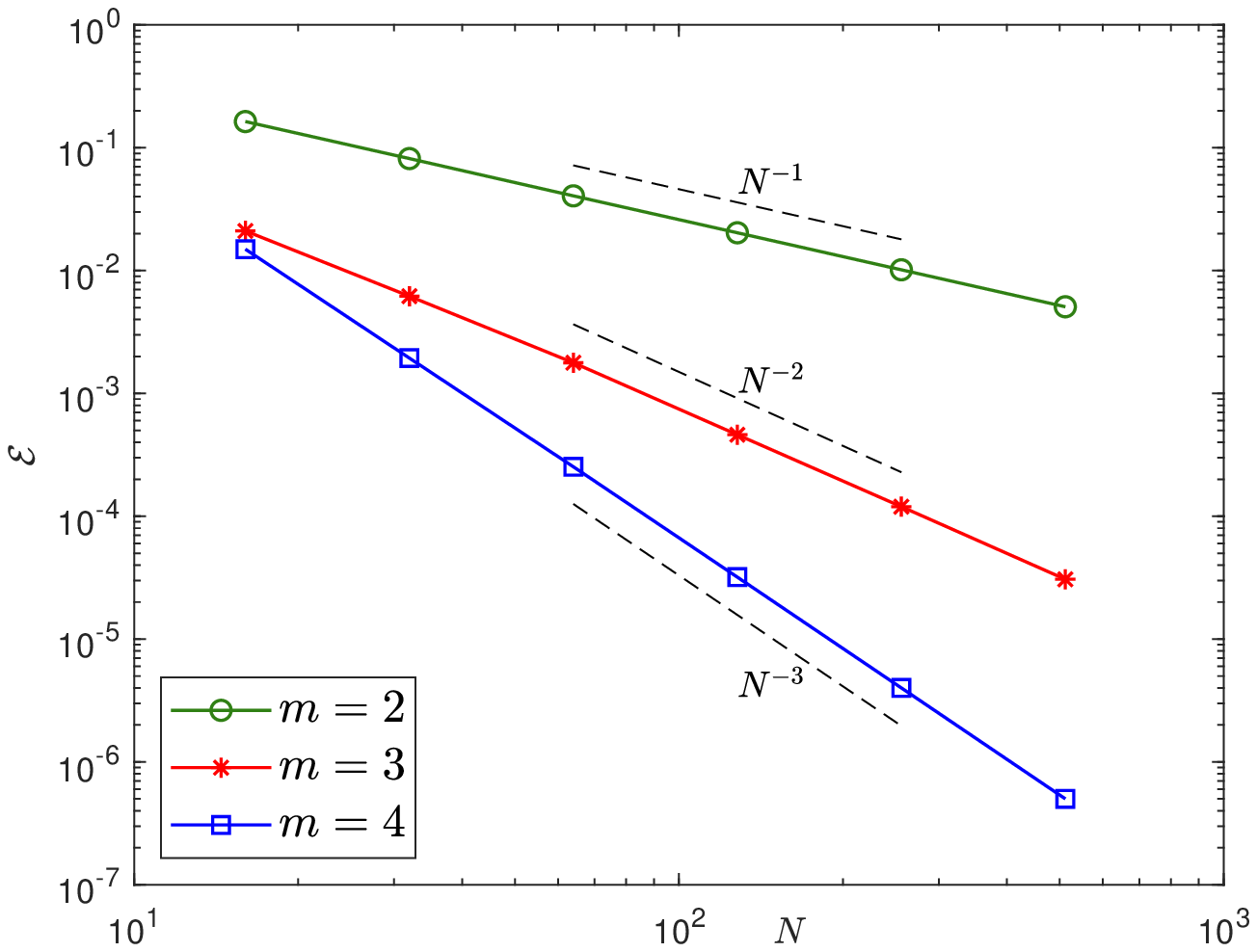} \quad
\includegraphics[scale=0.50]{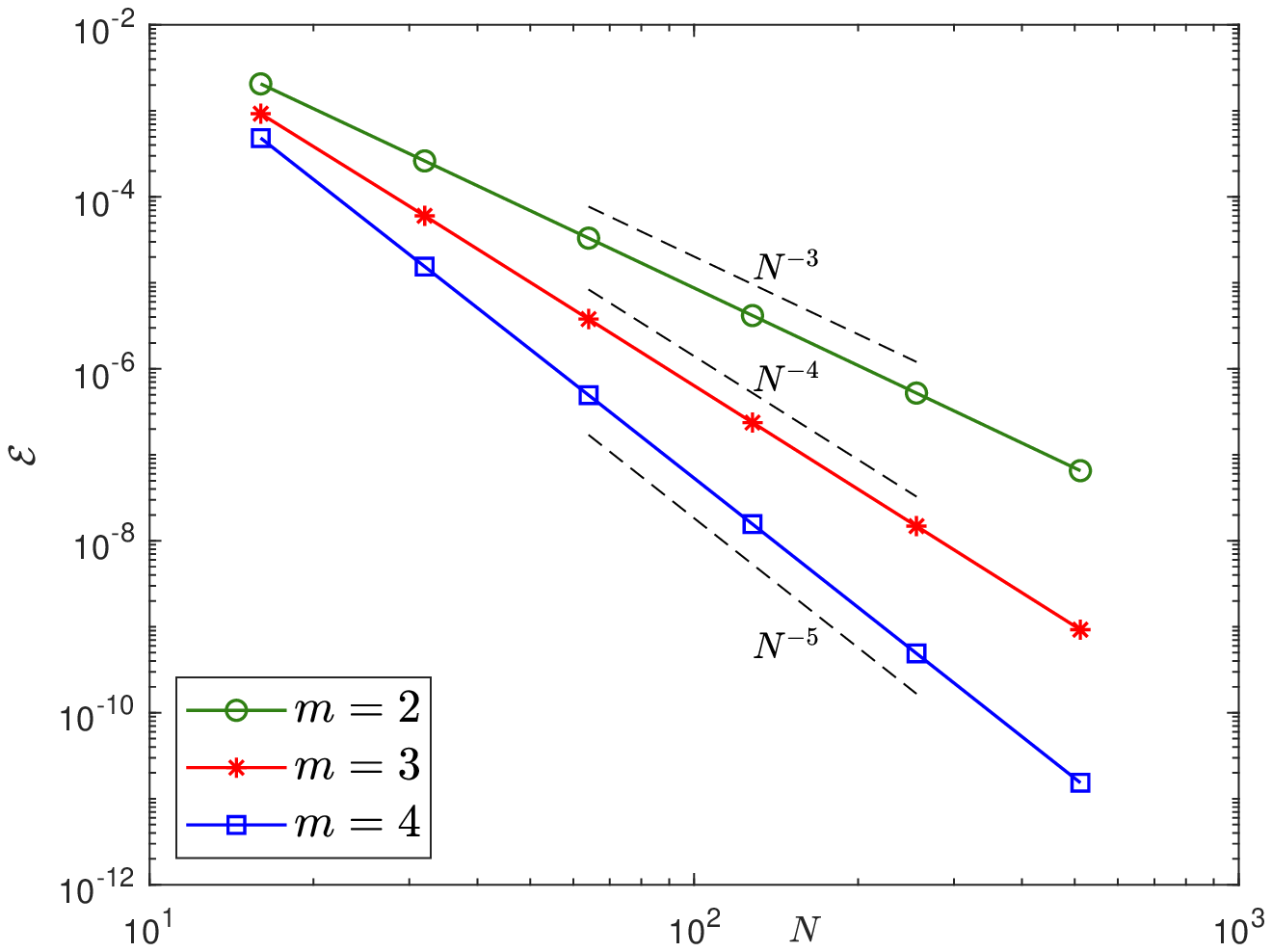} \\
\centering
\includegraphics[scale=0.50]{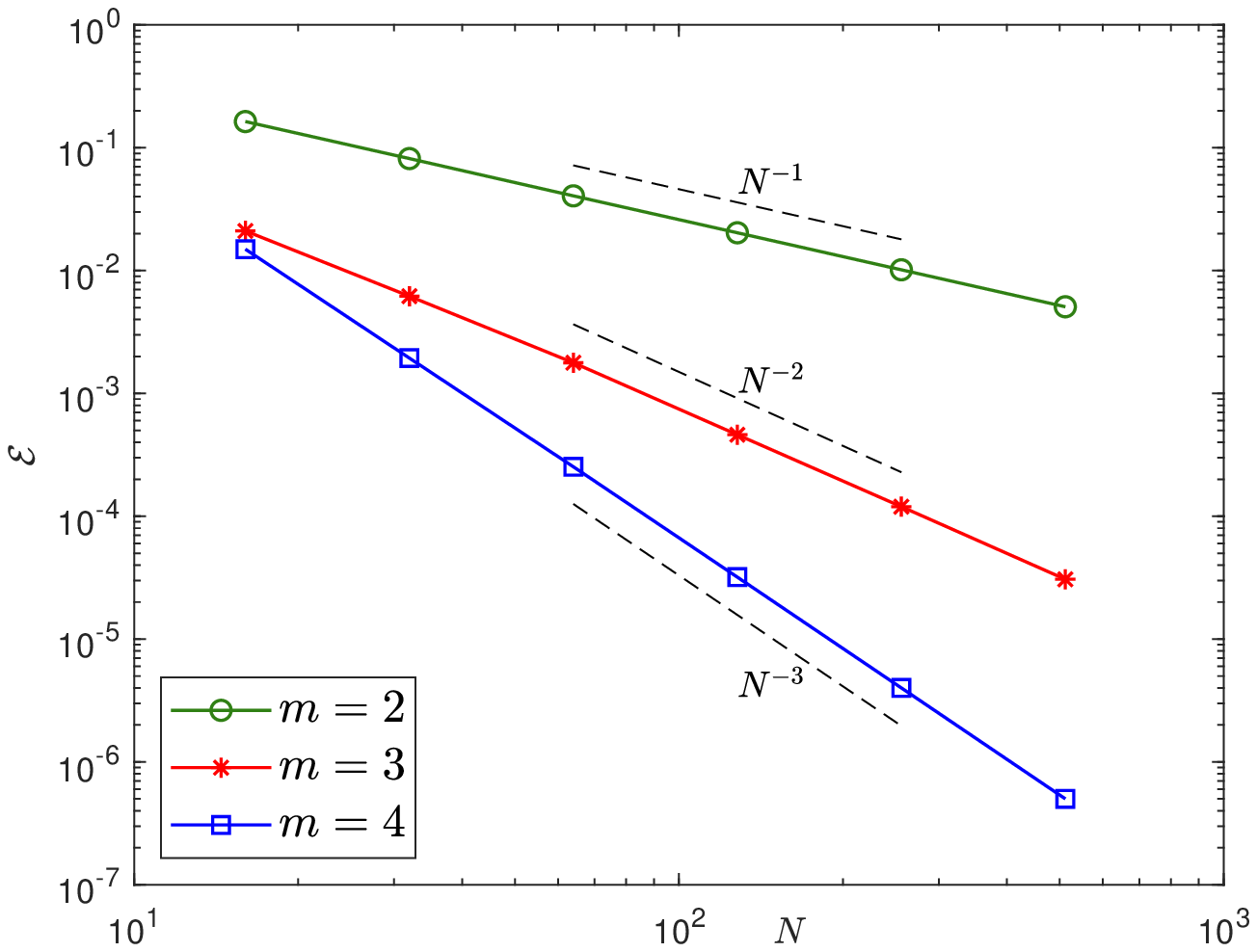} \quad
\includegraphics[scale=0.50]{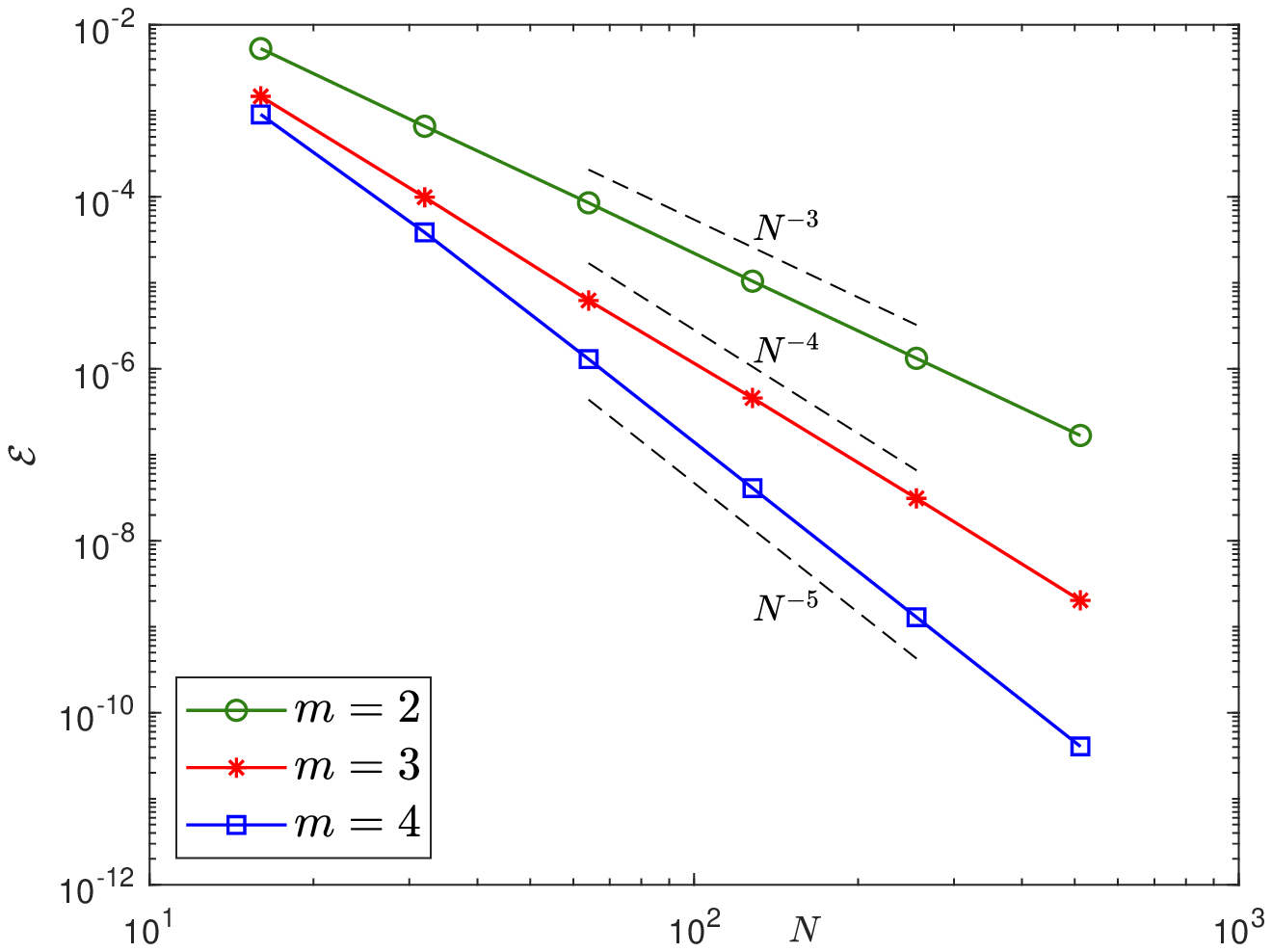} \\
\centering
\includegraphics[scale=0.50]{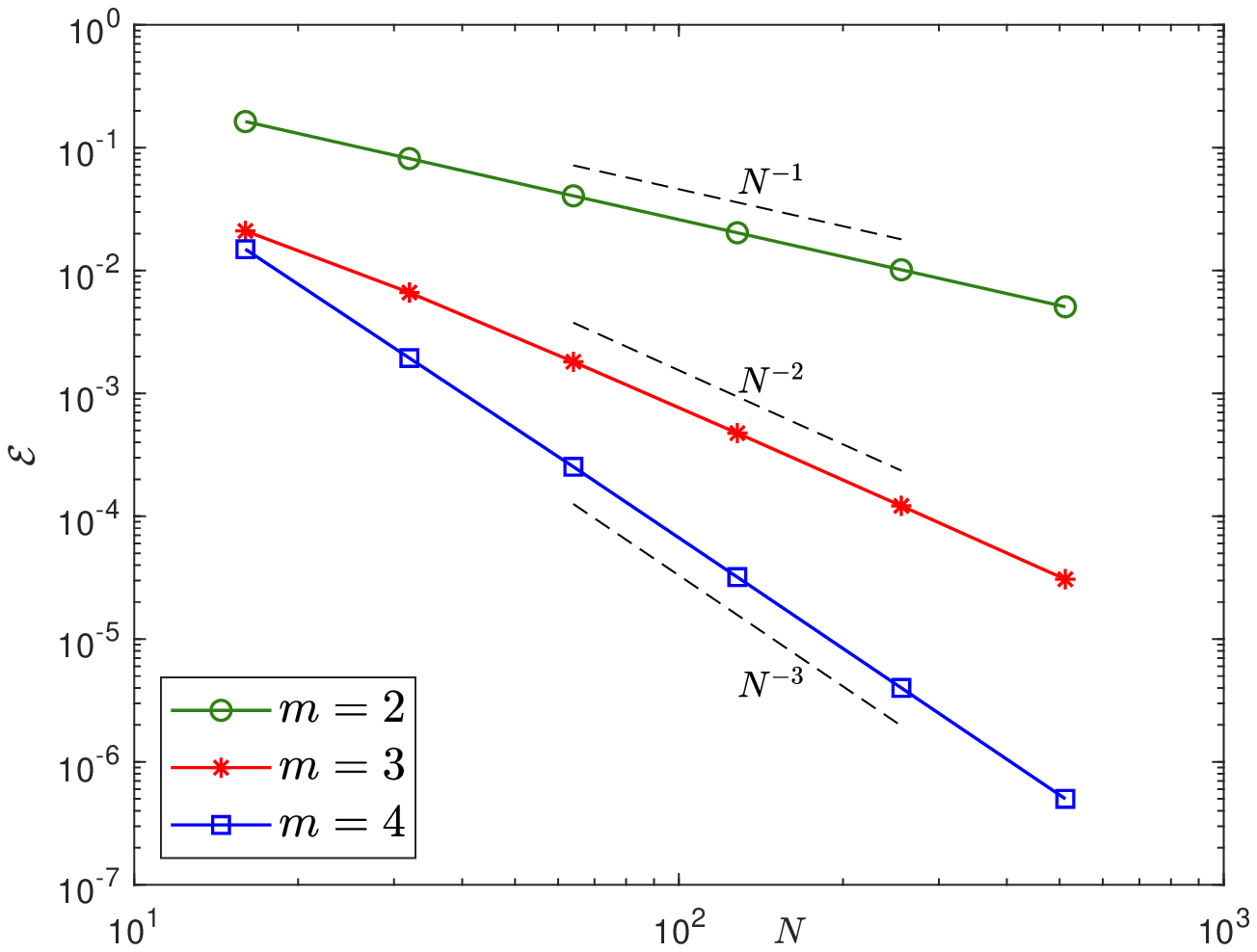} \quad
\includegraphics[scale=0.50]{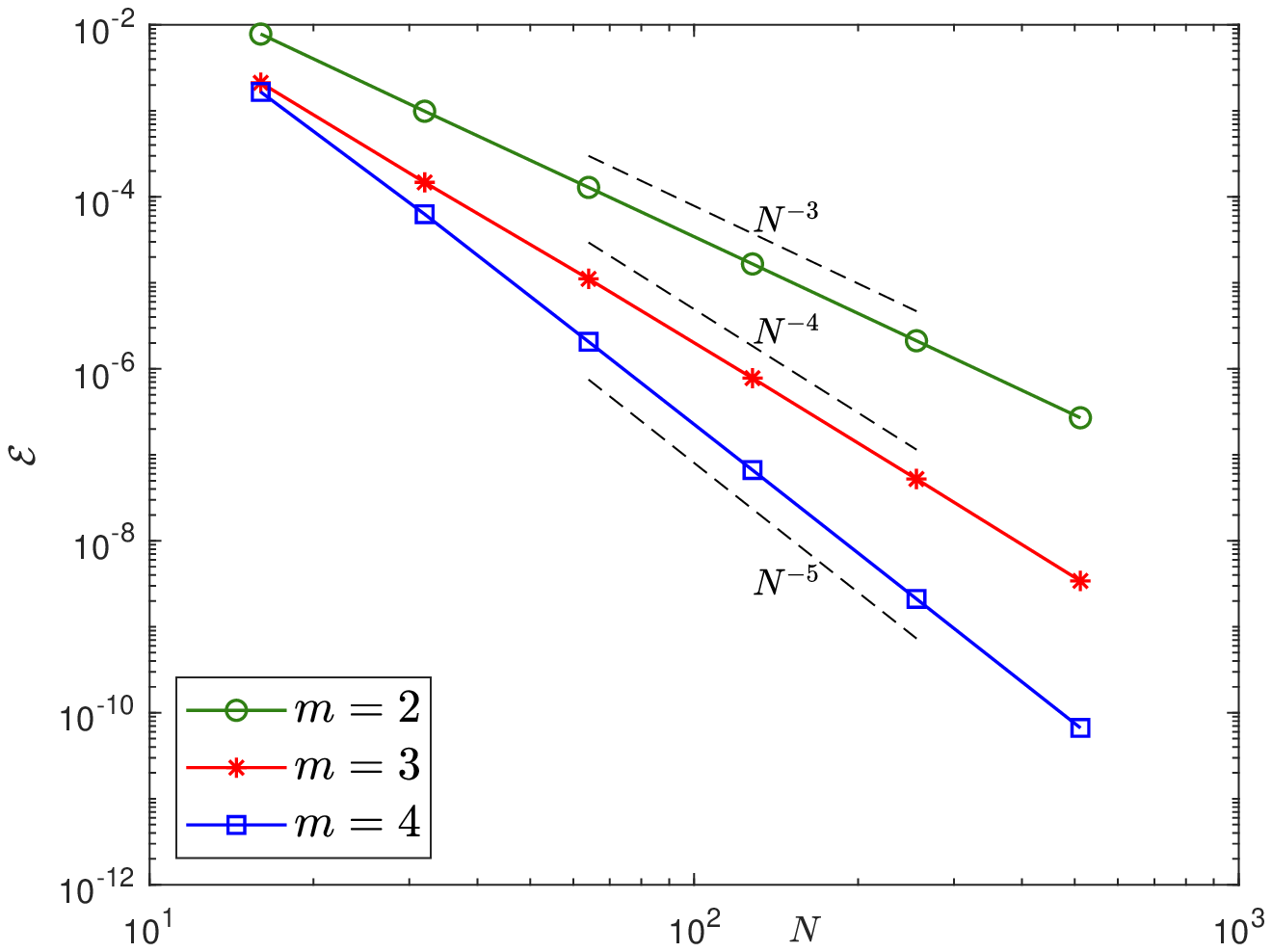} \\
\caption{Errors of $\partial_x \rho $ (left) and  $\partial_x \Phi $ (right) with $d=1, 2, 3$ (from top  to bottom) for Poisson potential with compact but  non-smooth density in \textbf{Example} \ref{Exp:nonsmooth}.}
\label{PoissonNotSmooth}
\end{figure}

\begin{figure}
\centering
\includegraphics[scale=0.50]{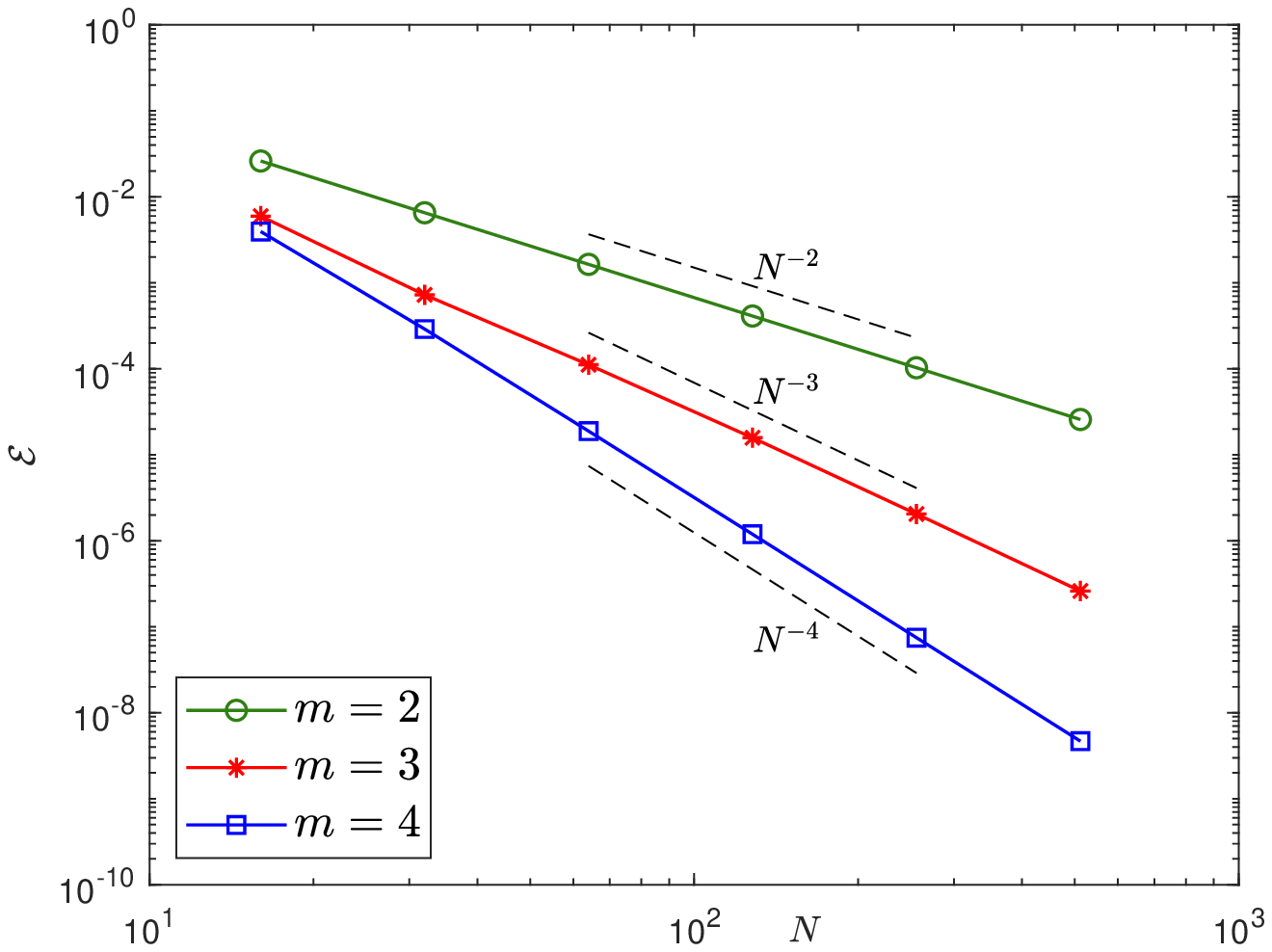} \quad
\includegraphics[scale=0.50]{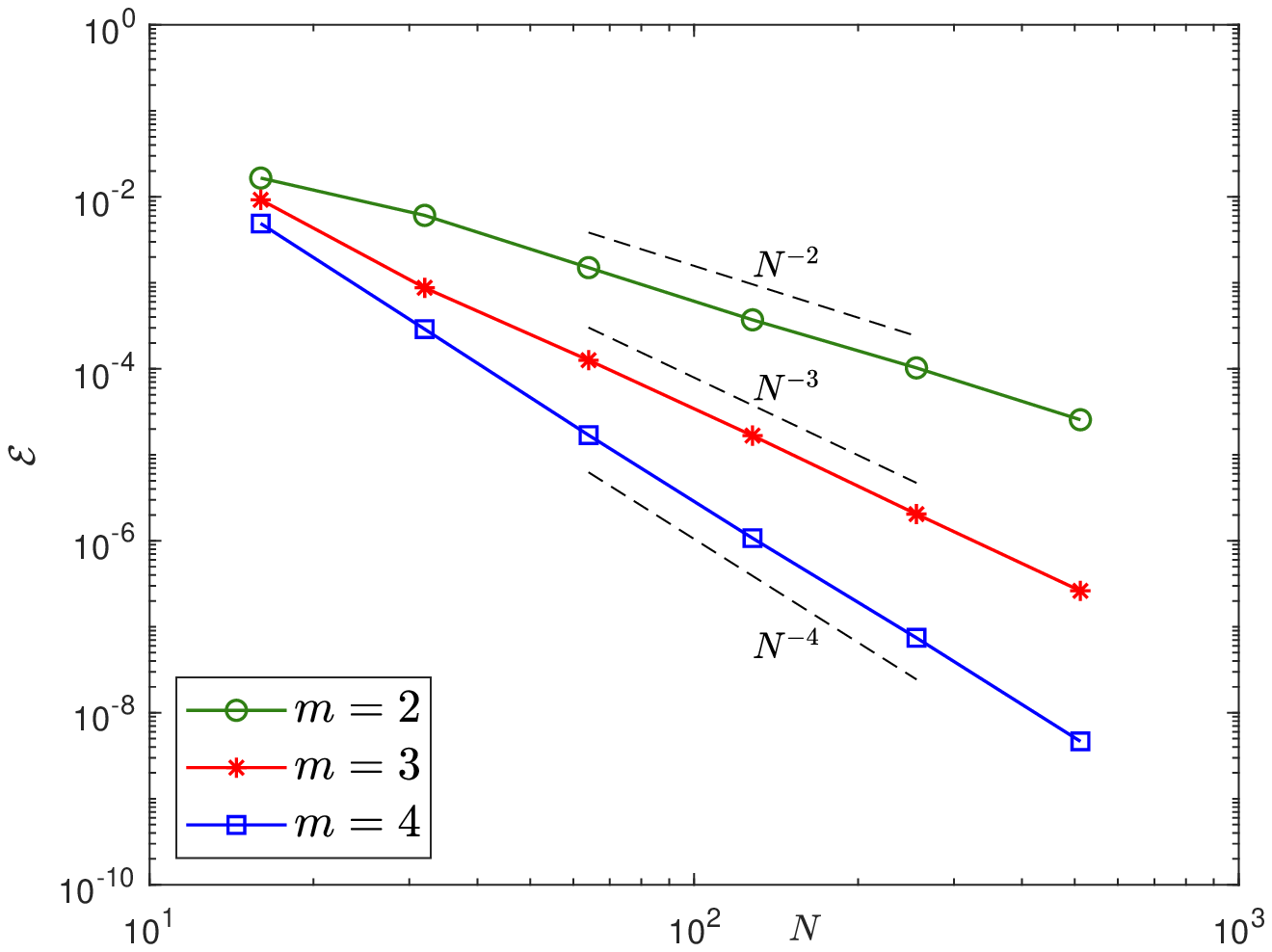}
\caption{Errors of $\partial_x \Phi $ for the 2D Coulomb potential with non-smooth density with $\bm{\delta}=(0,0)$(left) and $\bm{\delta}=(0.11, 0.22)$ (right) in \textbf{Example} \ref{Exp:nonsmooth}.}
\label{CoulombNotSmooth}
\end{figure}

\section{Conclusions} \label{conclusions}
In this article, we investigated the optimal zero-padding factor of the Kernel Truncation Method in the convolution-type nonlocal potential evaluation, and derived
the optimal zero-padding factor is $\sqrt{d}+1$ where $d$ is space dimension. The memory cost is greatly decreased to a small fraction, i.e., $(\frac{\sqrt{d}+1}{4})^d$,
of that in the empirical \textit{fourfold }algorithm.
We conclude that \textbf{twofold} padding is optimal for the 1D problem,
and \textbf{threefold }, instead of \textit{fourfold}, padding suffices to achieve spectral accuracy the 2D and 3D problems.
While for the anisotropic density, the total optimal zero-padding factor grows linearly with the anisotropy strength $\gamma_f$.
The KTM is realized with Fourier spectral method and it can be rewritten as a discrete convolution regardless of the anisotropy strength.
As long as the discrete tensor is available via a precomputation step, the effective computation is merely a pair of FFT and iFFT on a double-sized density,
and this is of significant importance in practical simulation, especially when the potential evaluation is called multiple times with the same setups.
Rigorous error estimates were provided for both the nonlocal potential and the density in $d$ dimension.
Extensive numerical results were presented to confirm the improvement in both memory and computational costs.

\section*{Acknowledgements}
Q. Tang was partially supported by the National Natural Science Foundation of China (No. 11971335)  and the Institutional Research Fund from Sichuan University (No. 2020SCUNL110).
X. Liu, S. Zhang and Y. Zhang was partially supported by the National Natural Science Foundation of China (No.12271400).
This work was done while Y. Zhang and S. Zhang were visiting Sichuan University in 2022.

\appendix
\section{Error estimates of Fourier spectral methods} \label{AppFourier}
\renewcommand{\theequation}{A.\arabic{equation}}
\setcounter{equation}{0}
For simplicity, we will consider $2 \pi$-periodic functions that are defined on torus $\Omega=\mathbb T^d:= [0,2\pi]^d$,
the extension to general domain is quite straightforward with only minor modifications.
Here we use $H_p^m(\Omega)$ to denote the $2\pi$-periodic function space with semi-norm given below
$$|u|_{m,\Omega} = \left(\sum_{ |\bm{\alpha}|=m} \|\partial^{\bm{\alpha}} u \|^2_{2, \Omega }\right)^{1/2}
\sim \left( \sum\limits_{\bk \in \mathbb{Z}^d} |\bk|^{2m}|\widehat{u}_{\bk}|^2 \right)^{1/2}.$$
Define the Fourier and physical index set respectively
\beas
&&\mathcal{I}_N := \left\{ \bk=(k_1, \dots, k_d)  \in \mathbb{Z}^d ~\big| ~ |k_p| \le N, p=1, \dots, d \right\},\\
&&\mathcal{T}_N := \left\{ \bx_{\textbf{j}}=(x_{j_1}, \dots, x_{j_d}) ~\big|~ x_{j_p} = j_p\pi/N, j_p = 0,\ldots, 2N-1 \right\}.
\eeas
The Fourier interpolation series at mesh grid $\mathcal T_N$ reads explicitly as
   \beas
u_{N}(\bx) := \sum_{\bk \in \mathcal{I}_N } \widetilde{u}_{\bk} ~e^{\im \bk \cdot \bx}  \text{ with} ~~
\widetilde{u}_{\bk}= \frac{1}{(2 N)^d c_{\bk}} \sum_{\bx_\textbf{j}\in \mathcal{T}_{N}} u(\bx_{\textbf{j} })e^{-\im \bk \cdot \bx_{\textbf{j}}},
\eeas
where $c_{\bk} : = \prod\limits_{p=1}^{d}c_{p}$ with $c_{p} = 1$ for $|p| < N$, and  $c_{p} = 2$ for $|p| = N$.
The error function $ u-u_{N} $ is characterized by the following lemma, which is an extension of the 1D results in \cite{ShenBook},
\begin{lem} \label{FouApprox}
 For any $u(\bx) \in H^{m}_p(\Omega)$  with $m> d/2$, we have
\beas
\| \partial^{\bm{\alpha}} (u- u_{N}) \|_{\infty,\Omega} & \lesssim & N^{-(m-\frac{d}{2}- |\bm{\alpha}|)}|u|_{m,\Omega} ,  \quad 0\leq |\alpha|\leq m,\\
\| \partial^{\bm{\alpha}}( u- u_{N} )\|_{2,\Omega} & \lesssim & N^{-(m - |\bm{\alpha}|)} |u|_{m,\Omega} ,\quad\quad ~0\leq |\alpha|\leq m.
\eeas

\end{lem}

\begin{proof}
Let $ P_N u$ be the orthogonal projection of $u$, which is also named as the truncated Fourier series, i.e.,
\beas
(P_N u)(\bx) = \sum\limits_{\bk \in \mathcal{I}_{N}} \widehat{u}_{\bk} e^{\im \bk \cdot \bx} ~~ \text{with} ~~ \widehat{u}_{\bk} = \frac{1}{(2 \pi)^{d}} \int_{\Omega} u(\bx) e^{-\im \bk \cdot \bx}  { \rm d} \bx.
\eeas
Using Cauchy-Schwarz inequality, we obtain
\bea\label{Pn-u}
|(P_N u -u)(\bx)| & \le &  \sum_{ \bk \not \in \mathcal{I}_{N} } |\hat{u}_{\bk}| \le  \sum_{|\bk|>N} |\hat{u}_{\bk}|
\le   \left(\sum_{|\bk|>N} |\bk|^{-2 m}\right)^{1/2} \left( \sum_{|\bk|>N} |\bk|^{2 m} |\hat{u}_{\bk}|^2\right)^{1/2} \notag \\
& \lesssim &  \left( \int_{|\bx| \ge N }  |\bx|^{-2 m} {\rm d} \bx \right)^{1/2} |u|_{m,\Omega}  \lesssim N^{-(m-\frac{d}{2})}   |u|_{m,\Omega}.
\eea
Define the interior of $\mathcal I_N$ as $\mathring{\mathcal{I}}_N:= \left\{ \bk=(k_1, \dots, k_d)  \in \mathbb{Z}^d ~\bigg| ~ |k_p| < N , ~ p=1, \dots, d  \right\}$
and the boundary set $\partial\mathcal{I}_N:=\mathcal{I}_N\setminus\mathring{\mathcal{I}}_N$.
It is clear that
\bea\label{Pn-un}
|(P_N u- u_{N})(\bx)| &\le& \sum_{\bk \in \mathcal{I}_{N} } |\hat{u}_{\bk} - \tilde{u}_{\bk}| = \sum_{ \bk \in \mathring{\mathcal{I}}_N} |\hat{u}_{\bk} - \tilde{u}_{\bk}| + \frac{1}{c_{\bk}}\sum_{\bk \in \partial\mathcal{I}_N }  | c_{\bk}\hat{u}_{\bk} - c_{\bk} \tilde{u}_{\bk}| \notag\\
&\le & \sum_{\bk \in \mathcal{I}_{N} } |\hat{u}_{\bk} -c_{\bk} \tilde{u}_{\bk}| + \sum_{\bk \in \partial\mathcal{I}_N}  (c_{\bk}-1) |\hat{u}_{\bk}|.
\eea
The last summation is bounded and estimated as follows
\beas
\sum_{\bk \in \partial\mathcal{I}_N}  (c_{\bk}-1) |\hat{u}_{\bk}| \lesssim  \sum_{\bk \in \partial\mathcal{I}_N} |\hat{u}_{\bk}| \le  \left(\sum_{\bk \in \partial\mathcal{I}_N}  |\bk|^{-2m}  \right)^{1/2}~ \left( \sum_{\bk \in \partial\mathcal{I}_N}  |\bk|^{2m} |\hat{u}_{\bk}|^2 \right)^{1/2} \lesssim  N^{-(m-\frac{d-1}{2})} |u|_{m,\Omega}.
\eeas
The following identity, connecting interpolation and projection function, holds for any function $u(\bx)\in H^m_p(\Omega)$
$$c_{\bk} \tilde{u}_{\bk}=\sum_{\textbf{s}\in \mathbb{Z}^d }\hat{u}_{\bk+2\textbf{s}N}.$$
The first summation \eqref{Pn-un} is estimated as
\bea\label{nodus}
\sum_{ \bk \in \mathcal{I}_{N}} |\hat{u}_{\bk} -c_{\bk} \tilde{u}_{\bk}| & = & \sum_{\bk \in \mathcal{I}_{N}} \big| \sum_{|\textbf{s}|>0}  \hat{u}_{\bk+2 \textbf{s} N} \big|  \notag\\
&\le& \left( \sum_{ \bk \in \mathcal{I}_{N}} \sum_{|\textbf{s}|>0}|\bk+2 \textbf{s} N |^{-2 m} \right)^{1/2} ~ \left( \sum_{ \bk \in \mathcal{I}_{N}} \sum_{|\textbf{s}|>0}  |\bk+2 \textbf{s} N |^{2 m} |\hat{u}_{\bk+2 \textbf{s} N}|^2 \right)^{1/2} \notag\\
&\lesssim & \left( \sum_{ \bk \in \mathcal{I}_{N}} \sum_{|\textbf{s}|>0}|\bk+2 \textbf{s} N |^{-2 m} \right)^{1/2} ~|u|_{m,\Omega},
\eea
where the double-summation is bounded from above
\beas
\sum_{ \bk \in \mathcal{I}_{N}} \sum_{|\textbf{s}|>0}|\bk+2 \textbf{s} N |^{-2 m}
\lesssim \mathop{\sum_{\textbf{q}\in \mathbb{Z}^d, |\textbf{q}|\ge N}} |\textbf{q}|^{-2m}  \lesssim    \int_{|\bx| \ge {N-1} }  |\bx|^{-2 m} {\rm d} \bx \lesssim  N^{-(2m-d)}.
\eeas
Then, we obtain
\bea
\sum_{ \bk \in \mathcal{I}_{N}} |\hat{u}_{\bk} -c_{\bk} \tilde{u}_{\bk}| \lesssim N^{-(m-\frac{d}{2})}|u|_{m,\Omega}.
\eea
Combine the above estimates and apply triangle inequality, we have
\beas
 | (u-u_{N})(\bx)| \le | (u -P_N u) (\bx)|+  |(P_N u - u_{N})(\bx)| \lesssim  N^{-(m-\frac{d}{2})} |u|_{m,\Omega},
\eeas
which, by taking maximum norm on both sides with respect to $\bx\in \Omega$, immediately leads to
\beas
 \|  u-u_{N} \|_{\infty, \Omega}  =    \sup_{\bx \in \Omega} | (u-u_{N})(\bx)|  \lesssim  N^{-(m-\frac{d}{2})}  |u|_{m,\Omega}.
\eeas

To estimate the $L^2$ norm, first we have
\bea\label{L2_1}
\|P_N u -u\|_{2,\Omega}^2  &= &  \sum_{ \bk \not \in \mathcal{I}_{N} } |\hat{u}_{\bk}|^2 \le  \sum_{|\bk|>N} |\hat{u}_{\bk}|^2 |\bk|^{2 m} |\bk|^{-2 m}     \le N^{-2 m}   |u|^2_{m,\Omega},\\[0.8em]
\label{L2_2}
\| P_N u -u_{N} \|_{2,\Omega}^2 &=& \sum_{\bk \in \mathcal{I}_{N} } |\hat{u}_{\bk} - \tilde{u}_{\bk}|^2 = \sum_{ \bk \in \mathring{\mathcal{I}}_{N}} |\hat{u}_{\bk} - \tilde{u}_{\bk}|^2 + \frac{1}{c^2_{\bk}}\sum_{\bk \in \partial\mathcal{I}_{N}}  | c_{\bk}\hat{u}_{\bk} - c_{\bk} \tilde{u}_{\bk}|^2 \notag \\
&\le & \sum_{\bk \in \mathcal{I}_{N} } |\hat{u}_{\bk} -c_{\bk} \tilde{u}_{\bk}|^2 + \sum_{\bk \in \partial\mathcal{I}_{N}}  (c_{\bk}-1)^2 |\hat{u}_{\bk}|^2  \notag\\
&\lesssim&  \sum_{\bk \in \mathcal{I}_{N} } |\hat{u}_{\bk} -c_{\bk} \tilde{u}_{\bk}|^2+N^{-2m} |u|^2_{m,\Omega}.
\eea
The summation in Eq.\eqref{L2_2} can be estimated by the Cauchy-Schwarz inequality as follows
\bea
\sum_{ \bk \in \mathcal{I}_{N}} |\hat{u}_{\bk} -c_{\bk} \tilde{u}_{\bk}|^2 & = & \sum_{\bk \in \mathcal{I}_{N}} | \sum_{|\textbf{s}|>0}  \hat{u}_{\bk+2 \textbf{s} N} |^2  \le \sum_{ \bk \in \mathcal{I}_{N}} \left(\sum_{|\textbf{s}|>0}|\bk+2 \textbf{s} N |^{-2 m} \right)~ \left( \sum_{|\textbf{s}|>0}  |\bk+2 \textbf{s} N |^{2 m} |\hat{u}_{\bk+2 \textbf{s} N}|^2 \right) \notag \\
&\le & \max_{ \bk \in \mathcal{I}_{N}} \left\{\sum_{|\textbf{s}|>0}|\bk+2 \textbf{s} N |^{-2 m} \right\} \sum_{ \bk \in \mathcal{I}_{N}}\left( \sum_{|\textbf{s}|>0}  |\bk+2 \textbf{s} N |^{2 m} |\hat{u}_{\bk+2 \textbf{s} N}|^2 \right) \notag \\
&\lesssim& \max_{ \bk \in \mathcal{I}_{N}} \left\{\sum_{|\textbf{s}|>0}|\bk+2 \textbf{s} N |^{-2 m} \right\} |u|_{m,\Omega}^2.
\eea
Since $|\bk|/N \leq \sqrt{d}$ holds for any $\bk \in \mathcal I_{N}$, we have
\bea\label{part2}
\sum_{|\textbf{s}|>0}|\bk+2 \textbf{s} N |^{-2 m} &=&(2N)^{-2m}\sum_{|\textbf{s}|>0}|\bk/2N +\textbf{s} |^{-2 m}  \leq C N^{-2m}, \eea
where the constant $C$ is independent of $N$, and it can be proved using the monotonicity argument via a $d$-dimensional integral.
It immediately suggests  $ \| P_N u -u_{N} \|_{2,\Omega}^2 \lesssim N^{-2 m}   |u|^2_{m,\Omega} $.
Therefore, we have
\beas
\| u -u_{N}\|_{2, \Omega}  \le \| u -P_{N} u\|_{2, \Omega} +\| P_{N} u -u_{N}\|_{2, \Omega}     \lesssim N^{-m}   |u|_{m,\Omega}.
\eeas

\

To estimate the spectral approximation error of its derivatives, following similar argument as shown above, we have
\beas
\| \partial^{\bm{\alpha}}(u- u_{N}) \|_{\infty,\Omega} & \lesssim &  \| \partial^{\bm{\alpha}}u- P_N(\partial^{\bm\alpha}u) \|_{\infty,\Omega}
+ \| P_N(\partial^{\bm\alpha}u)-\partial^{\bm{\alpha}}u_{N}) \|_{\infty,\Omega} \\
&\lesssim& N^{-(m-\frac{d}{2}- |\bm{\alpha}|)}|u|_{m,\Omega} ,
\eeas
and
\beas
\| \partial^{\bm{\alpha}}(u- u_{N}) \|_{2,\Omega} & \lesssim &  \| \partial^{\bm{\alpha}}u- P_N(\partial^{\bm\alpha}u) \|_{2,\Omega}
+ \| P_N(\partial^{\bm\alpha}u)-\partial^{\bm{\alpha}}u_{N}) \|_{2,\Omega} \\
&\lesssim& N^{-(m-|\bm{\alpha}|)}|u|_{m,\Omega}.
\eeas
\end{proof}

\renewcommand{\theequation}{B.\arabic{equation}}
\setcounter{equation}{0}
\section{Exact nonlocal potentials} \label{exact}
In this appendix, we present analytical results for the Poisson/Coulomb potential generated by the following radially symmetric non-smooth densities
\be
\rho(\bx)=\left\{\begin{array}{ll}
(1-|\bx|^2)^m,&|\bx| \le 1, \\[0.2em]
0, &|\bx| >1,\\
\end{array}\right.
\ee
where $\bx\in \mathbb R^d$ and $m\in \mathbb Z^{+}$. We denote $r  = |\bx|$ and the total mass of the density by $M(\rho):= \int_{\mathbb R^d} \rho(\bx) {\rm d}\bx$.
It is clear that the corresponding potential is also radially symmetric.

\

\textbf{The $d$-dimensional Poisson potential.} The Poisson potential $\Phi$ satisfies the following equation
\beas
-\Delta \Phi = \rho, \quad \bx \in \mathbb R^d.
\eeas
The potential can be derived by solving an exterior and interior Poisson problem respectively \cite{ExtBd}.
The 3D Poisson potential reads as
\be
\Phi(\bx)=\left\{\begin{array}{ll}
      c + \sum_{j=1}^{m+1} \frac{(-1)^{j}}{(2 j )(2 j + 1)} {\binom{m}{j-1}} r^{2j}, &r \le 1, \\[0.5em]
 M(\rho)\frac{1}{4\pi r}, & r > 1,
\end{array}\right.
\ee
with $c= \frac{M(\rho)}{4\pi} + \sum_{j=1}^{m+1} \frac{(-1)^{j-1}}{(2 j )(2 j + 1)} {\binom{m}{j-1}}$.
While the 2D Poisson potential satisfies the following far-field boundary condition\cite{ExtBd}
\be
\lim_{|\bx| \rightarrow \infty} \left[\Phi(\bx)+ M(\rho)~\frac{1}{2 \pi} \ln(|\bx|)\right]=0.
\ee
and we can derive the exact solution as follows
\be
\Phi(\bx)=\left\{\begin{array}{ll}
c + \sum_{j=1}^{m+1} \frac{(-1)^{j}}{(2 j )^2}{\binom{m}{j-1}}~r^{2j} , &r \le 1, \\[0.5em]
-M(\rho)\frac{1}{2 \pi} \ln(r), &  r > 1,
\end{array}\right.
\ee
where $c= \sum_{j=1}^{m+1} \frac{(-1)^{j-1}}{(2 j )^2}{\binom{m}{j-1}}$ is determined to ensure the continuity condition at $r=1$.
The 1D Poisson potential is given explicitly as
\be
\Phi(\bx)=\left\{\begin{array}{ll}
      c_{1} + c_2 r  + \sum_{j=1}^{m+1} \frac{(-1)^{j}}{(2 j-1)(2 j )} {\binom{m}{j-1}} r^{2j},&r \le 1,  \\[0.5em]
      -M(\rho)\frac{r}{2},  &r > 1,
\end{array}\right.
\ee
where $c_1 = \int_{\mathbb R} - \rho(x) \frac{1}{2}|x|{\rm d} x$ and $c_{2}= -M(\rho)/2 -c_1 -  \sum_{j=1}^{m+1} \frac{(-1)^{j}}{(2 j-1)(2 j )} {\binom{m}{j-1}} = 0$ .

\

\textbf{The 2D Coulomb potential corresponds to $U(\bx) = \frac{1}{2\pi |\bx|}$.}
For the 2D Coulomb potential, computing its Fourier integral, we obtain
\beas
\Phi(\bx) = \frac{1}{2 \pi} \int_0^{\infty} \widehat{U}(k) \widehat{\rho}(k) k J_0 (k r) ~
{\rm d}~ k, \quad r = |\bx|,
\eeas
and it can be integrated exactly. Here below we just present analytic results for some fixed $m\in \mathbb Z^{+}$
and the rest can be derived using mathematical software, like Mathematica.

For $m=2$, the 2D Coulomb potential reads as
\be
\Phi(\bx)=\left\{\begin{array}{ll}
      \frac{16  }{225 \pi r} \left[ r^2 p_0(r) E(\frac{1}{r^2})+ p_1(r) K(\frac{1}{r^2})\right] ,&r \leq 1, \\[0.5em]
\frac{16  }{225 \pi } \left[p_0(r) E(r^2)+p_2(r) K (r^2)\right], &r>1,
\end{array}\right.
\ee
where $p_0(r) = 23-23 r^2 + 8 r^4, ~p_1(r)=15-34 r^2+27 r^4-8 r^6,~p_2(r)=-4 (2-3 r^2+r^4) $, and
$K(r),E(r)$ are the complete elliptic integrals of first and second kind respectively\cite{handbook}.
For $m=3$, it is given explicitly as
\be
\Phi(\bx)=\left\{\begin{array}{ll}
-\frac{32  }{3675 \pi r} \left[r^2 l_0(r)E(\frac{1}{r^2})+ l_1(r) K(\frac{1}{r^2})\right] ,& r\leq 1, \\[0.5em]
-\frac{32 }{525 \pi } \left[l_0(r)E(r^2)+l_2(r) K (r^2)\right], &r>1, \\
\end{array}\right.
\ee
with $l_0(r)=8 (-22+33 r^2 -23 r^4+6 r^6) $, $l_1(r)=71-142 r^2+95 r^4-24 r^6$ and $l_2(r)= -105+298 r^2-353r^4+208 r^6-48r^8$.


\end{document}